\documentclass[reqno]{amsart}

\usepackage{amsmath,amssymb,amsfonts,amsthm}
\usepackage[dvipsnames]{xcolor}
\usepackage{multirow}
\usepackage{enumerate}
\usepackage{todonotes}

\DeclareFontFamily{OT1}{pzc}{}
\DeclareFontShape{OT1}{pzc}{m}{it}{<-> s * [0.900] pzcmi7t}{}
\DeclareMathAlphabet{\mathscr}{OT1}{pzc}{m}{it}

\usepackage{graphicx}
\usepackage{grffile}

\usepackage[utf8]{inputenc}

\usepackage[nonotes]{optional}

\newcommand{\arjun}[1]{\opt{notes}{\todo[size=\tiny,backgroundcolor=white,noline]{\textcolor{Mahogany}{arjun: #1}}}}

\newcommand{\notes}[1]{\opt{notes}{\textcolor{OliveGreen}{#1}}}

\usepackage[natbib=true,style=numeric,backend=biber]{biblatex}

\newcommand{\Secref}[1]{Section~\ref{#1}}
\newcommand{\Thmref}[1]{Theorem~\ref{#1}}

\newcommand{\Propref}[1]{Prop.~\ref{#1}}

\newcommand{\R}{\mathbb{R}} 
\newcommand{\Z}{\mathbb{Z}} 
\newcommand{\Q}{\mathbb{Q}} 
\newcommand{\Prob}{\mathbb{P}} 

\newcommand{\w}{\omega}

\newcommand{\norm}[2]{\left\vert{#1}\right\vert_{#2}} 
\newcommand{\grad}{\nabla} 
\newcommand{\AND}{\textrm{ and }} 
\newcommand{\OR}{\textrm{ or }}

\newcommand{\supp}{\operatorname{supp}} 
\newcommand{\almostsurely}{\textrm{a.s}}

\renewcommand{\limsup}{\varlimsup}
\renewcommand{\liminf}{\varliminf}

\DeclareMathOperator{\argmax}{argmax}

\newcommand{\Uniform}{\operatorname{Uniform}}

\newcommand{\Bernoulli}{\operatorname{Bernoulli}}

\newcommand{\la}{\langle}
\newcommand{\ra}{\rangle}

\newcommand{\E}{\mathbb{E}}

\renewcommand{\emph}[1]{\textbf{#1}}

\newcommand{\origin}{\mathbf{0}}

\newcommand{\vx}[0]{\vec{X}}

\newcommand{\mU}{\mathcal{U}}
\newcommand{\riU}{\operatorname{ri}(\mathcal{U})}

\newcommand{\gpp}[2][\beta]{g_{pp}(#1,#2)}
\newcommand{\overlinegpp}{\overline{g}_{pp}^{\beta}}
\newcommand{\gpl}[2][\beta]{g_{pl}(#1,#2)}

\newcommand{\mgf}{\mathbb{M}}
\newcommand{\mgfrw}{M}

\newcommand{\logmgf}{\mathbb{\Lambda}}
\newcommand{\logmgfrw}{\Lambda_p}

\newcommand{\gammatwo}{\mathbb{\Gamma}_2}

\newcommand{\entropy}{\mathcal{H}}

\newcommand{\relativeHqp}{\entropy(\mathbb{Q}_\beta|\mathbb{P})}

\newcommand{\weak}{\operatorname{WEAK}}
\newcommand{\strong}{\operatorname{STRONG}}

\newcommand{\rwF}{F}
\newcommand{\weightsF}{\mathcal{F}}

\theoremstyle{plain}
\newtheorem{thm}{Theorem}[section]
\newtheorem*{thm*}{Theorem}
\newtheorem{lem}[thm]{Lemma}
\newtheorem*{lem*}{Lemma}

\newtheorem{prop}[thm]{Proposition}
\newtheorem*{prop*}{Proposition}

\theoremstyle{definition}
\newtheorem{define}[thm]{Definition}

\newtheorem*{question*}{Question}

\newtheorem*{example*}{Example}
\newtheorem{claim}[thm]{Claim}

\theoremstyle{remark}
\newtheorem{rem}[thm]{Remark}
\newtheorem*{rem*}{Remark}

\title{On the phase diagram of the polymer model}

\author[A.\,Krishnan]{Arjun Krishnan}
\author[S.\,Mkrtchyan]{Sevak Mkrtchyan}
\author[S.\,Neville]{Scott Neville}
\address{Arjun Krishnan\\ University of Rochester\\  Department of Mathematics\\ Rochester, NY 14627\\ USA.}
\email{arjun@shirleyarjun.net}
\urladdr{https://people.math.rochester.edu/faculty/akrish11/}
\address{Sevak Mkrtchyan\\ University of Rochester\\  Department of Mathematics\\ Rochester, NY 14627\\ USA.}
\email{sevak.mkrtchyan@rochester.edu}
\urladdr{https://people.math.rochester.edu/faculty/smkrtchy/}
\address{Scott Neville\\ University of Michigan\\ Department of Mathematics\\ Ann Arbor, MI 48109\\USA.}
\email{nevilles@umich.edu}
\thanks{A.\,Krishnan was partially supported by a Simons Collaboration grant 638966.}
\thanks{S.\,Mkrtchyan was partially supported by a Simons Collaboration grant 422190.}
\thanks{S.\,Neville was partially supported by a NSF RTG 1840234}


\usepackage{bbold}

\DeclareSymbolFontAlphabet{\amsmathbb}{AMSb}%

\usepackage[urlcolor=blue,colorlinks=true,linkcolor=blue,citecolor=blue]{hyperref}

\begin{document}
\begin{abstract}
In dimensions 3 or larger, it is a classical fact that the directed polymer model has two phases: Brownian behavior at high temperature, and non-Brownian behavior at low temperature. We consider the response of the polymer to an external field or tilt, and show that at fixed temperature, the polymer has Brownian behavior for some fields and non-Brownian behavior for others. In other words, the external field can \emph{induce} the phase transition in the directed polymer model.
\end{abstract}
\maketitle

\setcounter{tocdepth}{2}
\tableofcontents

\section{Introduction}
\label{sec:introduction}
The directed polymer model was introduced by \citet{huse_henley_1985} to study the influence of inhomogeneities on interfaces in the two dimensional Ising model. Nowadays, it is more commonly viewed as a model of polymer chains in a solution with impurities. The equilibrium shape of the polymer is influenced by both the (attractive) impurities and by the random thermal fluctuations in the medium. The strength of the thermal fluctuations relative to the impurities is controlled by an inverse temperature parameter $\beta$.

The classical paradigm in this model is the existence of two distinct phases or behavior regimes. In the ``weak disorder'' phase, the thermal fluctuations overwhelm the influence of the impurities, and the polymer behaves diffusively like a random walk or Brownian motion. When the impurities dominate and the polymer is nondiffusive, it is said to be in ``strong disorder''. In fact, in one space dimension, the polymer model is always in strong disorder, and has the universal KPZ scaling~\cite{MR2917766} and fluctuation behavior~\cite{MR3116323,krishnan_universality_2016,MR4341071}. 

The transition between phases may also be described in terms of the free energy $F(\beta)$ of the polymer. The critical temperature $\beta_c \in [0,\infty]$ is the smallest $\beta$ at which $F(\beta)$ fails to be an analytic function. When $\beta<\beta_c$, the polymer is said to be in the high-temperature phase, and when $\beta>\beta_c$, in the low temperature phase. In one or two space dimensions, $\beta_c$ is known to be $0$. In dimensions three or higher, the polymer experiences a sharp phase transition at some $\beta_c > 0$. Under some technical restrictions, the weak/strong-disorder phases and the high/low-temperature phases are now known to be synonymous. 

In other classical models displaying a phase transition, there are two functions called the Gibbs and Helmholtz free energies that are extremely useful in not only defining phase transitions like in the previous paragraph, but also in extracting averages of various statistical quantities under the Gibbs measure. For example, in the Ising model, the Helmholtz free energy $F^H(\beta,h)$ is a function of $\beta$ and $h$, the external field. Its convex dual in the $h$ parameter is the Gibbs free energy $F^G(\beta,m)$, where $m$ is the average magnetization under the Gibbs measure. In $d=1$, the model was solved by Ising \citep{ising_beitrag_1925} in his thesis, and shows no phase transition. However, in $d=2$, the model was famously solved by Onsager \citep{MR0010315} when $h=0$, and does display a phase transition. The phase transition appears as a singularity in the derivative of the free energy $\lim_{h \to 0^\pm} \partial F^H(\beta,h) = \pm m_0$, where $m_0$,  the average zero-field magnetization is only nonzero below the critical temperature. This is a prototypical example of a symmetry breaking phase transition. The Ising model remains more-or-less unsolved for $h \neq 0$ and $d > 2$, and describing its two-parameter free energies and phase diagram is an open problem of considerable interest.

The directed polymer model and the closely related random walk in random environment (RWRE) are special cases of the more general random walk in random potential (RWRP). There, typically 2 free energies are considered  ---the point-to-point and point-to-level free energies--- that have an additional directional and field parameter respectively. The field parameter appears in the Hamiltonian just like the external field does in the Ising model. As a result, the point-to-point and point-to-level free energies are simply the classical Gibbs and Helmholtz free energies of the polymer model.  In this paper, we consider the behavior of the standard directed polymer model at fixed temperature, and show that the polymer can be in weak disorder for some external fields, while being in strong disorder for others. In other words, in contrast to the Ising model, the field can \textit{induce} the phase transition from Gaussian to non-Gaussian behavior in the polymer model at fixed temperature.

The first mathematically rigorous formulation of the directed polymer model was by Imbrie and Spencer \cite{MR968950}. Much of the early work can be found in \cite{MR1939654,MR2199799,MR1006293,MR2271480,MR1413246,MR2249671,MR2299713}; see \cite{MR2073332} and \cite{MR3444835} for comprehensive reviews. Let the measure space $(\R^{\Z_{\geq 0} \times \Z^d},\mathcal{F},\Prob)$ represent a field of iid weights $\{\w_{t,x}\}_{t \in \Z_{\geq 0}, x \in \Z^d}$ on the lattice $\Z^d$ independently for each time-step $t \in \Z_{\geq 0}$ (the nonnegative integers). We assume that the moment generating function $\mgf(\beta) = \E[ e^{\beta \w_{t,x}} ]$ is finite for all $\beta \in \R$, where $\E$ represents expectation with respect to the measure $\Prob$. For a finite lattice path $\gamma \colon \{0,\ldots,n\} \to \Z^d$, we define the weight or energy of the path as
\notes{The condition $\mgf(\beta) < \infty$ for \textbf{all} $\beta \in \R$ is what is used in Comets, in both his book and in Proposition 2.5 of \cite{MR1996276}}
\begin{equation}
    H(\gamma) := \sum_{i=1}^n \w_{i,\gamma(i)}.
\end{equation}

Consider a random walk $\vec{X}:=(X_n)_{n=0}^\infty$ in $\Z^d$, where the single-step probabilities are given by the probability mass function $p\colon \Z^d \to [0,1)$. The expectation and measure on the space of walks $((\Z^d)^{\Z_{\geq 0}},F)$ where $F$ is the natural $\sigma$-field, are $E_{p}$ and $P_{p}$ respectively. The measures $P_p$ and $\Prob$ are independent.
The partition function of the model is defined by
\begin{equation*}
    Z_{n,\beta,p}(x,y) := E_{p} \left[ \exp( \beta H(\vx_n)) | X_0 = x, X_n = y \right],
\end{equation*}
where $\vx_n$ refers to $(X_i)_{i=0}^n$, the first $n$ steps of the random walk path. We write $Z_{n,\beta,p}(x)$ instead of $Z_{n,\beta,p}(\origin,x)$, where $\origin$ is the origin in $\Z^d$; and we write 
$$Z_{n,\beta,p} = E_{p} [\exp( \beta H(\vx_n)) | X_0 = \origin]$$ when the endpoint is not constrained. The partition function is used to define the Gibbs measures on paths $\vx_\infty$, where the probability distribution of a path is given by 
\begin{equation}
    \mu_{n,\beta,p}(d\vx_\infty) := \frac{ e^{ \beta H(\vx_n)} }{Z_{n,\beta,p}}P_p(d \vx_\infty).
    \label{eq:gibbs measure definition}
\end{equation} 
That is, the first $n$ steps of the walk are given the Gibbs weight, and thereafter, the walk uses the transition probability of the kernel $p$. By projecting this measure on the first $n$ steps, the probability of a path $\vx_n$ is
\begin{equation*}
    \frac{ e^{ \beta H(\vx_n)} }{Z_{n,\beta,p}}P_p(\vx_n),
\end{equation*} 
which, abusing notation, we again denote by $\mu_{n,\beta,p}$. 
Given a functional $f$ on the space of random walks of length $n$, we write the expectation of $f$ with respect to the Gibbs measure as
\begin{equation*}
    \la f \ra_{n,\beta,p} := E_{p} [ f(\vx_n) \mu_{n,\beta,p}(\vx_n) ].
\end{equation*}
In the following, we will sometimes not include $\beta \AND p$ in the subscript as long as they are fixed and understood, and write $Z_n \OR \mu_n$ instead.

\subsection{Classical results}

This section discusses classical results and can be skipped by experts. 

Under the assumption of finite exponential moments, it is a classical theorem \cite[Proposition 2.5]{MR1996276} that the limit
\begin{equation}
    \lim_{n \to \infty} \frac{\log( Z_{n,\beta})}{n} =: F(\beta) 
\end{equation}
exists $\Prob$ a.s and in $L^1$. The constant $F(\beta)$ is called the \textit{free-energy} of the model, and from Jensen's inequality, it follows that
\begin{equation}
   F(\beta) \leq \logmgf(\beta),
    \label{eq:free energy annealed bound}
\end{equation}
where $\logmgf(\beta)=\log \mgf(\beta)$ is the logarithm of the moment generating function of the weights. This is called the \emph{annealed bound} for the free-energy. There are two classical phases of the polymer, essentially defined by whether or not equality holds in \eqref{eq:free energy annealed bound}. When equality holds, the polymer is said to be in the \textbf{high-temperature} phase, and when equality does not hold, the polymer is said to be in the \textbf{low- temperature} phase (also called the \emph{very strong disorder} phase in \cite{MR2579463}).  

Another important quantity is the \textit{normalized} partition function
\begin{equation}
   W_n := \frac{ Z_{n,\beta,p} }{ \mgf(\beta)^n},
   \label{eq:normalized partition function martingale}
\end{equation}
originally introduced by Bolthausen in his pioneering work \citep{MR1006293}. This is a martingale with respect to the filtration $\mathcal{F}_n = \sigma( \{ \w_{i,x} \}_{i \in \{0,\ldots,n\}, x \in \Z^d} )$, which corresponds to knowing all the weights up to time $n$. Since $W_n$ is also nonnegative,  
\begin{equation*}
    \lim_{n \to \infty} W_n = W_\infty,\quad \Prob-\almostsurely,
\end{equation*}
and by the Kolmogorov $0-1$ law, we have the dichotomy $\Prob( W_\infty > 0) = 1$ or $\Prob(W_\infty = 0) = 1$. The polymer is said to be in \textbf{weak disorder} if the former holds, and in \textbf{strong disorder} when the latter holds. A similar critical temperature $\overline{\beta_c}$ can be defined that separates the weak and strong disorder regimes. It is not difficult to see that weak disorder implies that we must have $\beta < \beta_c$; i.e., we must be in the high-temperature phase. It follows that $\beta_c \geq \overline{\beta_c}$, and in fact, $\beta_c = \overline{\beta_c}$ was a well-known conjecture \citep[Open Problem 3.2]{MR3444835}, that now appears to be solved \citep{junk_strong_2024}. 

The standard second-moment method provides a sufficient condition for weak disorder. It involves obtaining a uniform in $n$ bound for $\E[W_n^2]$. Since $W_n$ is a martingale, this shows that $W_n \to W_\infty$ almost surely and in $L^2$ and thus $\Prob( W_\infty > 0) = 1$.  
\begin{thm}[\cite{MR1006293,MR1413246}]
   Let 
   \begin{equation}
       \pi(p) := P_p( S_n = S'_n \text{ for some } n),
       \label{eq:intersection probability of two copies of the random walk}
   \end{equation}
   where $S_n$ and $S_n'$ are independent copies of the random walk with transition kernel $p$. If $\pi(p) < 1$ and
   \begin{equation}
       \mgf_2(\beta) := \frac{\mgf(2\beta)}{\mgf(\beta)^2} < \frac{1}{\pi(p)},
       \label{eq:L2 regime defining condition}
   \end{equation}
   then $W_\infty > 0$ a.s. 
   \label{thm:original l2 theorem for polymer}
\end{thm}
The set $\{ \beta \colon \mgf_2(\beta) < \pi^{-1}(p)\}$ is called the $L^2$ region of the polymer. In this region, it is known that we are in weak disorder; i.e., $\beta < \overline{\beta_c}$, and in fact, the diffusively rescaled polymer path $\vx_n$ under the Gibbs measure converges to a Brownian motion in probability. More generally, for $p>1$, one can similarly define the region where the martingale $W_n$ is bounded in $L^p$. Let 
$$
    p^*(\beta) = \sup\{ p \colon (W_n)_{n=1}^\infty \text{ is } L^p \text{ bounded} \}.
$$
Under the assumption of bounded weights, \citet{Junk2022} showed that in the weak disorder region, $p^*(\beta) \geq 1 + 2/d$, and conjectured \citep{FukushimaJunk2023} that $p^*(\beta) > 1 + 2/d$ for $\beta < \beta_c$. It is known that $p^*(\beta) \geq 2$ in the $L^2$ region of weak-disorder. A relaxation of the bounded weights assumption are given in \citep{FukushimaJunk2023}. 

Theorem \ref{thm:original l2 theorem for polymer} is generally stated and proved in the case where $p$ is the standard nearest-neighbor random walk on $\Z^d$, but it is obvious and well-known that it applies to reasonably general transition kernels $p$. 

While Theorem \ref{thm:original l2 theorem for polymer} provides a sufficient condition for weak-disorder, it doesn't prove the existence of a critical temperature. This was proved by \citet{MR2271480}, who relied on an innovative application of the FKG inequality. We combine and summarize several results from Bolthausen to Comets-Yoshida in the theorem below. 
\begin{thm}[\cite{MR968950,MR1006293,MR1413246,MR2271480,MR1371075}]
    \hphantom{A}
    \begin{itemize}
        \item There is a critical value $\overline{\beta_c} \in [0,\infty]$ such that weak disorder holds when $\beta < \overline{\beta_c}$, and strong disorder holds when $\beta > \overline{\beta_c}$.
        \item A phase transition also occurs in the free energy: $\exists \beta_c \in [0,\infty]$ such that $F(\beta) = \logmgf(\beta)$ for $\beta \leq \beta_c$ and $F(\beta) < \logmgf(\beta)$ for $\beta > \beta_c$.
        \item In general $\overline{\beta_c} \leq \beta_c$.
        \item The critical value $\beta_c$ is $0$ for $d=1,2$ and $\overline{\beta_c} \in (0,\infty]$ for $d \geq 3$.
        \item If $d \geq 3$ and $\beta \in [0,\overline{\beta_c}]$, the rescaled polymer path under the Gibbs measure converges to a Brownian motion in probability.
    \end{itemize}
   \label{thm:full characterization of weak disorder bolthausen to comets et al}
\end{thm}
 
Theorem \ref{thm:full characterization of weak disorder bolthausen to comets et al} doesn't guarantee that strong disorder or a low temperature regime exist; i.e., that $\overline{\beta_c}$ or $\beta_c$ are finite. This is guaranteed by the following entropic condition. If there are two measures $Q \ll P$, then the \textit{relative entropy} of the two measures is given by
\begin{equation}
    \entropy(P | Q) := - \int \log \left( \frac{dQ}{dP} \right) dP.
    \label{eq:relative entropy of two measures}
\end{equation}
Let $\mathbb{Q}_\beta(d\w_\origin) := \exp(\beta \w_\origin) \mgf(\beta)^{-1} \Prob(d\w_\origin)$. That is,  $\mathbb{Q}_\beta$ is a Gibbs-biased distribution on one individual weight. 
\begin{prop}
    Let $p$ be the standard nearest-neighbor random walk on $\Z^d$. If $\relativeHqp > \log(2d)$, then $F(\beta) < \logmgf(\beta)$; i.e., the polymer is in the low-temperature regime.
    \label{prop:condition on gamma 2 for strong disorder}
\end{prop}
The condition in Proposition \ref{prop:condition on gamma 2 for strong disorder} was originally introduced by Kahane and Peyri\'ere \cite{MR0431355} in a different context, and is recorded in \citep{MR1996276} in the polymer setting. In the literature, the quantity $\relativeHqp$ is called $\gammatwo(\beta)$ and has the formula
\begin{equation}
    \relativeHqp = \beta \logmgf'(\beta) - \logmgf(\beta).
    \label{eq:gammatwo quantity used in strong disorder condition}
\end{equation}
We find the relative entropy interpretation more appealing for reasons that will become clear (see Theorem \ref{thm:strong disorder statement}), and will prefer it over $\gammatwo(\beta)$.  

In weak disorder, the polymer endpoint under the Gibbs measure is spread out on the $\sqrt{n}$ scale and is distributed like a Gaussian. In strong disorder, the situation is quite different, and the polymer endpoint is concentrated around favored spots, a phenomenon called ``localization''. One way to observe this phenomenon is to study the probability of the most likely endpoint, namely
$$J_n:=\max_z\mu_{n-1,\beta,p}(X_n=z).$$
When the polymer is in weak disorder and the endpoint is diffusive, $J_n$ decays as a power of $n$ \citep[Remark 1.13]{junk_local_2023}. But when the endpoint localizes, $J_n$ will be relatively large; more precisely, the polymer is said to be localized \citep[Definition 5.1]{MR3444835} if 
\begin{equation}
    \liminf_{n\to\infty}\frac 1n\sum_{t=1}^n J_t>0,\ \mathbb{P}-\text{a.s.},
    \label{eq:localization definition}
\end{equation}
and delocalized otherwise. For the standard directed polymer with nearest-neighbor steps and Gaussian weights, \citet{MR1939654} proved $\limsup_n J_n>0$
in $d=1 \AND 2$ for all $\beta > 0$. This was extended to the stronger statement \eqref{eq:localization definition} and generalized to all iid weights with finite moment generating function by \citet{MR1996276}. We refer to \citep{MR1939654,MR1996276,MR4089496} for more details and generalizations. 

\section{The external field in the polymer model}
\subsection{Gibbs and Helmholtz Free Energies} 
As mentioned in the introduction, the directed polymer model can be seen as a RWRP \citep[Example 1.1]{rassoul-agha_quenched_2013,rassoul-agha_quenched_2014}. In those articles, more general potentials $V$ that depend on finite path segments (instead of just the vertex the path passes through) were considered, but we restrict to iid vertex-weight potentials (called local) here. Let $\supp(p) = \{ z \colon p(z) > 0 \}$ be the set of steps the random walk can take, and let $\mU$ be the convex hull of $\supp(p)$. All of their results are restricted to the case where $p$ has finite-range; i.e., $|\supp(p)| < \infty$.  

In RWRP, there is no explicit time-coordinate, and so, to cast our directed walks in their setting, we embed our walks in $d+1$ dimensions with the $e_0$ dimension representing time, and choose the kernel of our walk to satisfy $p(x) > 0$ only if $x = e_0 + \sum_{i=1}^d a_i e_i$ for some integers $a_i$. Then, our walks would be directed in their sense, and as long as $\E[ \w_0^{p} ] < \infty$ for some $p > d$, we are guaranteed the almost-sure existence of the \emph{point-to-point} and \emph{point-to-level} limiting free energies:
\begin{align}
    \gpp{x} & := \lim_{n \to \infty} \frac{1}{\beta n} \log E_p \left[ \exp( \beta H(\vx_n) ) | X_0 = 0, X_n = [nx] \right] \quad \forall x \in \mU, \label{eq:gpp definition} \\
    \gpl{h} & := \lim_{n \to \infty} \frac{1}{\beta n} \log E_p \left[ \exp( \beta H(\vx_n) + \beta X_n \cdot h ) | X_0 = \origin \right] \quad \forall h \in \R^{d}, \label{eq:gpl definition}
\end{align}
where $[x]$ denotes the unique lattice point in $[x, x+1)^{d+1}$. Under the assumption of $\mgf(\beta) < \infty$, the existence of the limit in \eqref{eq:gpl definition} is just the classical result of Comets, Shiga and Yoshida \citep[Proposition 2.5]{MR1996276}. The limits in \eqref{eq:gpp definition} and \eqref{eq:gpl definition} originally appear in \citep[Theorem 2.3]{rassoul-agha_quenched_2013} and \citep[Theorem 2.2]{rassoul-agha_quenched_2014} and apply to a wide class of ergodic potentials. They also prove several variational formulas for both $g_{pp}$ and $g_{pl}$. \citet{MR3556777} provide an interesting and mysterious characterization of the weak disorder regime in terms of these variational formulas. \citet{MR3535900} extend these results to zero temperature $(\beta = \infty)$, and delve deeper into the analysis of the minimizers of the variational formulas. 

It is clear that the set $\mU \subset \{ e_0 + x \colon x = \sum_{i=1}^d c_i e_i \}$ for $c_i \in \R$ and thus $\gpp{\cdot}$ and $\gpl{\cdot}$ are really only $d$ dimensional functions of the parameters $(c_1,\ldots,c_d)$. Hence, we view $\gpp{\cdot}$ and $\gpl{\cdot}$ as functions on $\R^d$ and not as functions on $\R^{d+1}$ as defined above. Similarly, we will view $\mU$ and $\supp(p)$ as subsets of $\R^d$ and $\Z^d$ respectively, and $p(x)$ as a kernel on $\Z^d$. It is easy to see 
that \citep[eq. (4.3)]{MR3535900}
\begin{equation}
    \gpl{h} = \sup_{x \in \mU} \{ \gpp{x} + h \cdot x \}, \quad h \in \R^d.
    \label{eq:gpl in terms of gpp}
\end{equation}
In other words, $\gpl{h}$ is the convex dual of the convex function $-\gpp{x}$. A priori, $\gpp{x}$ may not be continuous up to the boundary of $\mU$, and hence $\gpp{x}$ has to be extended with an upper semicontinuous regularization to invert \eqref{eq:gpl in terms of gpp} and establish the usual convex duality between $\gpp{x}$ and $\gpl{h}$. See \citep[Section 4]{MR3535900} for details.  The tilt-velocity (field-direction in our language) duality is standard fare in large deviations theory, and appears in the rate function for the endpoint of RWRE and RWRP going back to \citet{MR1675027} and \citet{MR1989232}. In the case of the directed polymer, it probably first appears in \citep[equation (4.3)]{rassoul-agha_quenched_2014}.  

We have no need for $g_{pp}$ and the duality formula in this paper. However, we do need the existence of $g_{pl}$ when $|\supp(p)|=+\infty$. This is so that we can consider the case of the Gaussian random walk, which gives an example where the ``existence of strong disorder'' part of our theorems does not apply and for some temperatures the polymer is always in weak disorder (see Section \ref{sec:gaussian random walk}). The assumption of finite support of $p$ appears as an important technical restriction in the series of papers cited above. One of the technical contributions of this paper is to remove this restriction in the existence results for $\gpl{h}$. In fact, our results are stronger, since they show that the limit in \eqref{eq:gpl definition} exists simultaneously for all $\beta \in (0,\infty)$ and $h \in \R^d$, $\Prob-\almostsurely$ (see Proposition \ref{prop:existence and convexity of gpl}).

\notes{I think this section may not be necessary to include. The theory of convex duality is well developed when the convex functions involved are upper-semicontinuous. So it is of some value to establish when $\gpp{x}$ is continuous up to the boundary of $\mU$. A priori, we do not know this, and we therefore define as in \citep{MR3535900}, $\gpp{x} = -\infty ~ \forall x \in \mU^c$, and let $\overlinegpp(x) = \max( \gpp{x}, \limsup_{y \to x} \gpp{y})$ be the upper semicontinuous regularization of $\gpp{x}$ defined on all of $\R^d$. Then, by standard convex duality \citep{rockafellar_convex_1970}
\begin{align*}
    \label{eq:}
    \gpl{h} & = \sup_{x \in \R^d} \left\{ \overlinegpp(x) + h \cdot x \right\}\\
    \overlinegpp(x) & = \inf_{h \in \R^d} \left\{ \gpl{h} - h \cdot x \right\}\\
    \gpp{x} & = \inf_{h \in \R^d} \left\{ \gpl{h} - h \cdot x \right\} \quad x \in \riU
\end{align*}
where $\riU$ is the relative interior of $\mU$ and the last part follows because $\gpp$ is continuous on $\riU$. We say that $h$ and $x \in \riU$ are dual to each other if
\begin{equation*}
    \gpl{h} = \gpp{x} + h \cdot x
\end{equation*}
}

\subsection{The effect of the external field}
In the usual Ising model on a rectangle $R \subset \Z^d$, the Hamiltonian of a spin configuration $\sigma \in \{\pm 1\}^R$ under external field $h \in \R^d$ is 
\begin{equation*}
   H(\sigma) = \sum_{x \sim y, x,y \in R} \sigma_x \sigma_y + h \cdot \sum_{x \in R} \sigma_x,
\end{equation*}
where $x \sim y$ means $x$ and $y$ are nearest-neighbors. The partition function is given, as usual, by $Z = \sum_\sigma \exp( \beta H(\sigma))$. If we view the random walk in the polymer model as a sequence of increments $\sigma = (X_{i+1} - X_i)_{i=0}^{n-1}$, then the parameter $h$ appears in \eqref{eq:gpl definition} just like the external field appears in the Ising model, so we refer to the parameter $h$ as an \textbf{external field} that biases the underlying random walk. Other authors refer to $h$ as the \textit{tilt}. The field terminology also makes sense from the point of view of the central limit theorem of the polymer endpoint; see Remark \ref{rem:misalignment with underlying drift}.

The point-to-level limit can be written in terms of the original random walk with drift, as follows:
\begin{align}
    E_p & \left[ \exp\left(  \beta ( H(\vx_n) + h \cdot X_n)  \right) | X_0 = 0 \right]  \nonumber\\
    & = \mgfrw_p(\beta h)^n \sum_{\vx_n} \exp( \beta H(\vx_n) ) \prod_{i=1}^n \frac{ e^{\beta(X_i - X_{i-1} ) \cdot h}}{\mgfrw_p(\beta h)} p(X_i - X_{i-1}),\nonumber\\
    & = \mgfrw_p(\beta h)^n E_{q_{p,\beta}(h)} \left[ \exp( \beta H(\vx_n) ) | X_0 = \origin \right],\nonumber
\end{align}
where $\mgfrw_p(\beta h) = E_p[ \exp( \beta (X_1 - X_0) \cdot h ) ]$ is the moment generating function of the underlying random walk, and $q_{p,\beta}(h,\cdot)$ is a random walk transition kernel given by
\begin{equation}
    q_{\beta,p}(h,x) := \frac{ e^{\beta x \cdot h}}{\mgfrw_p(\beta h)} p(x).
    \label{eq:transition probability in direction h}
\end{equation}
To save ourselves from a notational alphabet soup, we will drop $p$ and $\beta$ from the notation, and simply refer to $q(h)$ as the random walk transition probability under the external field $h$. The field $h$ introduces a drift into the underlying random walk in direction $\beta^{-1} \grad_h \logmgfrw(\beta h)$, where $\logmgfrw(\beta h) = \log \mgfrw(\beta h)$. Then,
\begin{equation}
    \gpl{h} = \lim_{n \to \infty} \frac{1}{\beta n} \log Z_{n,\beta,q(h)} + \frac{1}{\beta}  \logmgfrw(\beta h) \quad \almostsurely .
\end{equation}
With a two parameter free energy $\gpl{h}$ like in the Ising model and other systems in statistical physics, it is of considerable interest to describe the phase diagram of the system. This is the main purpose of this paper.

Taking $\E$-expectation over the weights and applying Jensen's inequality, we obtain the annealed bound
\begin{equation}
    \gpl{h} \leq \frac1{\beta} \left( \logmgf(\beta) + \logmgfrw(\beta h) \right).
    \label{eq:annealed bound for gpl}
\end{equation}

By analogy with the annealed bound in \eqref{eq:free energy annealed bound} in \Secref{sec:introduction}, we say that the polymer is in the \textbf{high temperature} regime when the external field is $h$ if equality holds in \eqref{eq:annealed bound for gpl}, and in \textbf{low temperature} otherwise. Strictly speaking, it is the pair $(h,\beta)$ that is at high or low temperature, but when $\beta$ is fixed and understood, we will simply refer to the field $h$ as being in high or low temperature. Again, it is easy to check that
\begin{equation}
   W_{n,\beta}(h) := \frac{ Z_{n,\beta,q(h)} }{ \mgf(\beta)^n}
\end{equation}
is a nonnegative martingale (c.f.~\eqref{eq:normalized partition function martingale}), and thus has a $\Prob$-almost sure limit $W_{\infty,\beta}(h)$. As before, we can define \textbf{weak and strong disorder} under field $h$ according as $W_{\infty,\beta}(h) > 0$ or $W_{\infty,\beta}(h) = 0$, $\Prob-\almostsurely$. By the Kolmogorov 0--1 law, this is a dichotomy, and therefore, we define two deterministic sets of external fields
\begin{align*}
    \label{eq:}
    \weak_\beta & := \{ h \in \R^d \colon \Prob( W_{\infty,\beta}(h) > 0) = 1 \}, \\
    \strong_\beta & := \{ h \in \R^d \colon \Prob( W_{\infty,\beta}(h) = 0) = 1 \}. 
\end{align*}
It is clear that $\R^d$ is the disjoint union of $\weak_\beta \AND \strong_\beta$. However, it is not clear when each of them is nonempty.
 
\subsection{Main Results}
Our main results (Theorems \ref{thm:simultaneosly in weak and strong disorder} and \ref{thm:strong disorder statement}) shows that polymers in $d \geq 3$ and fixed $\beta > 0$ can simultaneously be in weak disorder for some values of the external field, and in strong disorder in others. For all of our results, we assume
\begin{equation}
    \mgf(\beta) < \infty \quad \forall \beta \in \R, \AND \mgfrw_p(t) < \infty \quad \forall t \in \R^d.
    \label{eq:assumptions on both mgfs}
\end{equation}

\begin{prop}
    The limit in \eqref{eq:gpl definition} exists almost surely simultaneously for all $\beta \in (0,\infty)$ and $h \in \R^d$, and defines $\gpl{h}$, which is a continuous function of $(\beta,h)$. $\gpl{h}$ is convex in $\beta$ and $h$ separately, and $\gpl{h/\beta}$ is jointly convex in $(\beta,h)$.
    \label{prop:existence and convexity of gpl}
\end{prop}
Proposition \ref{prop:existence and convexity of gpl} is a point-to-line analog of a full ``shape theorem'' for the free energy. The main difficulty here over the point-to-point version in \citep[Theorem 2.2]{rassoul-agha_quenched_2014}, is that $\supp(p)$ could be infinite. However, our assumption on the weights \eqref{eq:assumptions on both mgfs} is stronger than theirs, and as noted in \citep{MR2381596}, Proposition 2.5 in \citep{MR1996276} does work to show the existence of $g_{pl}$ on any dense countable subset of $(0,\infty) \times \R^d$ even when $\supp(p)$ is infinite. We provide a Lipschitz estimate on compact subsets for $\gpl{h}$ to prove Proposition \ref{prop:existence and convexity of gpl} in Section \ref{sec:shape theorem}.
\begin{rem}
    Note that from physical considerations, $\gpl{h}$ as defined in \eqref{eq:gpl definition}, is the natural Helmholtz free-energy of the system. However, in some respects, $\gpl{h/\beta}$ is the more natural object to study, because it separates out the effect of $\beta$ and $h$ in the partition function and would lead to cleaner statements in Proposition \ref{prop:existence and convexity of gpl} and Theorem \ref{thm:monotonicity}. That said, we work with $\gpl{h}$ as that is the object that was first introduced in the mathematics literature (\citep{rassoul-agha_quenched_2014}), it makes the duality between $\gpl{h}$ and $\gpp{x}$ in \eqref{eq:gpl in terms of gpp} cleaner, and is the prevalent choice in the physics literature as well.
\end{rem}

Let the Shannon entropy of the random walk transition probability be
\begin{equation}
    \entropy(p) := -\sum_{x \in \Z^d} \log(p(x)) p(x) = -E_p[ \log p(X_1) | X_0 = \origin ].
    \label{eq:shannon entropy of transition probability}
\end{equation}
For any $K \subset \supp(p)$, define 
\begin{equation}
    \entropy_K(p) := -E_p\left[ \log \left( \frac{p(X_1)}{p(K)} \right) \bigg| X_1 \in K, X_0 = \origin \right].
\end{equation}
That is, $\entropy_K(p)$ is the Shannon entropy of the random walk conditioned to take steps in the set $K$. 
\begin{thm}
   Let $\beta > 0$ be fixed.
   \begin{enumerate}
       \item \label{item:weak disorder statement} If $\mgf_2(\beta) < \pi(p)^{-1}$, where $\pi(p)$ is the intersection probability defined in \eqref{eq:intersection probability of two copies of the random walk}, $\weak_\beta$ contains a nonempty neighborhood of the origin such that $\mgf_2(\beta) < \pi(q(h))^{-1}$, and thus the polymer is at high temperature for any $h$ in this neighborhood.
       \item \label{item:strong disorder holds when p has finite range} If $p$ is nontrivial and has finite range, there exists a deterministic set $D \subset \R^d$, such that for each $h \in \R^d \setminus D$, $\lambda h$ is in low-temperature (and hence in  $\strong_\beta$) for all large enough $\lambda > 0$. $D$ is contained in a finite union of hyperplanes; see  \eqref{eq:defining the codimension one bad set of fields}.
   \end{enumerate}
   \label{thm:simultaneosly in weak and strong disorder}
\end{thm}
When $\mgf_2(\beta) < \pi(p)^{-1}$ and $p$ has finite range, Theorem \ref{thm:simultaneosly in weak and strong disorder} implies that for such $\beta$, $\weak_\beta$ and $\strong_\beta$ are both nonempty. See Figure \ref{fig:phase diagram} for a schematic of the phase diagram implied by our results. 

\begin{figure}
\includegraphics[width=5cm]{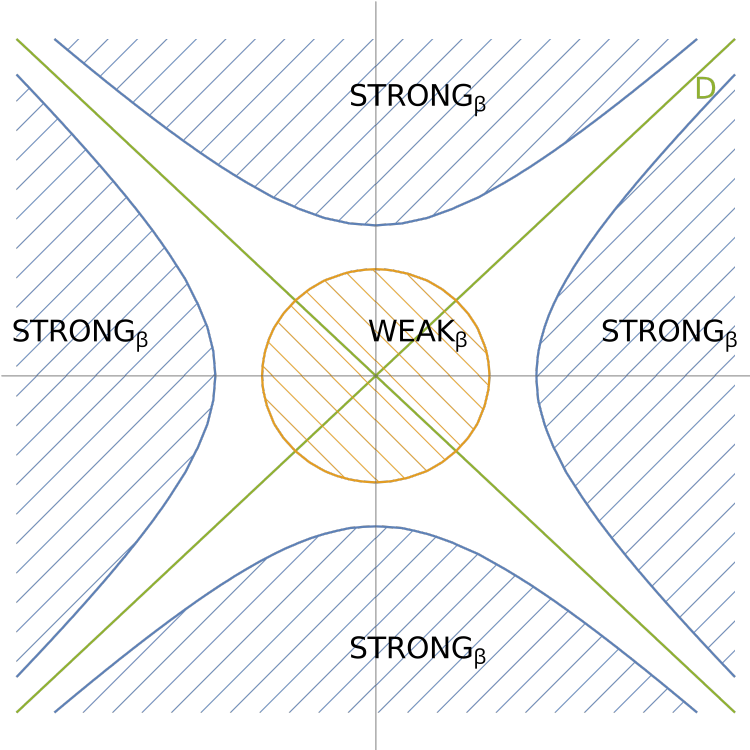}
\caption{
This is a schematic of the weak and strong disorder regions as a function of the external field $h\in\R^3$ on the plane $h_3=0$ when $d=3$, $p$ is the simple random walk and $\beta$ is small enough to be in the $L^2$ region. $D$ is contained in the diagonal lines $h_1 = \pm h_2$. As long as $\beta$ is small enough, there is region of weak disorder for small $|h|$. If $h$ is not in $D$, our results say that $\lambda h$ is at low-temperature and hence in $\strong_\beta$ for $\lambda$ large enough. Our results cannot precisely pin down the phase transition in $h$ with $\beta$ fixed; and so in the region between $\weak_\beta$ and $\strong_\beta$ is left unshaded and unpatterned.}
\label{fig:phase diagram}
\end{figure}

Item \ref{item:strong disorder holds when p has finite range} in Theorem \ref{thm:simultaneosly in weak and strong disorder} is actually fairly explicit; that is, we can identify $h$ for which the polymer is in strong disorder. 
\begin{thm}
    Let $\beta > 0$ and $h$ be such that 
    \begin{enumerate}
        \item $K(h) := \argmax_{x \in \supp(p)} h \cdot x$ is nonempty, and
        \item    
           \begin{equation}
               \relativeHqp > \entropy_{K(h)}(p),
           \label{eq:gammatwo condition for general random walks}
       \end{equation}
       where $\relativeHqp$ is defined in \eqref{eq:gammatwo quantity used in strong disorder condition}.
    \end{enumerate}
     Then for all large enough $\lambda > 0$, $\lambda h$ is at low temperature, and thus $\lambda h \in \strong_\beta$.
   \label{thm:strong disorder statement} 
\end{thm}
A standard condition implying strong disorder at fixed $(\beta,h)$ is $\relativeHqp > \entropy(q(h))$ (see Proposition \ref{prop:entropy condition for strong disorder}), which compares an entropy of the weights to the entropy of the underlying random walk. For $(\beta,\lambda h)$, as $\lambda \to \infty$, this condition transforms into \eqref{eq:gammatwo condition for general random walks}. Although the inequality in \eqref{eq:gammatwo condition for general random walks} looks hard-to-verify, it is trivial when $K(h)$ is a singleton, since the conditional Shannon entropy is $0$. The finite range assumption on $p$ guarantees that this happens for some fields $h$. For example, if $\supp(p) = \{z_1,\ldots,z_k\}$, it is easy to see that
\begin{equation}
    D \subset \R^d \setminus \{ h \colon |K(h)| = 1 \} \subset \bigcup_{1 \leq i < j \leq k} \{ h \colon h \cdot z_i = h \cdot z_j \}.
    \label{eq:defining the codimension one bad set of fields}
\end{equation}
Lemma \ref{lem:simultaneous weak and strong disorder for finite range walks} has the details.

\begin{rem}
    The intuitive reason for strong disorder is the following. Suppose $p$ is the standard simple random walk. The external field $h$ gives the walk a drift. When the drift is really large, for example, when $h = K e_1$ with $K \gg 1$, it makes the walk effectively zero dimensional; i.e., it makes it highly likely to just step in the $+e_1$ direction. In two dimensions or lower, the polymer is always in strong disorder (see Theorem \ref{thm:original l2 theorem for polymer}). The set $K(h)$ in Theorem \ref{thm:strong disorder statement} more or less determines the ``effective dimension'' of the walk with field $\lambda h$ as $\lambda \to \infty$. A stronger version of Theorem \ref{thm:strong disorder statement} ought to be true: when the ``effective dimension'' of the walk with field $h$ is $0, 1 \OR 2$, $\lambda h \in \strong_\beta$ for all large enough $\lambda$. On the other hand, if $h \in D$, $K(h)$ has sufficiently large dimensionality, and $\beta$ is small enough, it is possible that $\lambda h \in \weak_\beta$ for all $\lambda > 0$.
\end{rem}

When $p$ does not have finite range, for example, when $p(x) = C\exp( -|x|^2 )$ (discrete Gaussian), we can no longer implement our strategy to prove the existence of fields $h$ at low temperature. The bigger the drift you add to a Gaussian, the more it stays the same. So Theorem \ref{thm:strong disorder statement} fails to imply strong disorder since $\entropy(q(h)) \geq c > 0$ for all $h \in \R^d$ and therefore \eqref{eq:gammatwo condition for general random walks} does not hold for $\beta$ small enough. In fact, we can show more.
\begin{prop}
Let $p$ be the transition kernel of the discrete Gaussian random walk in $d
\geq 3$. For all $\beta$ small enough and all $h \in \R^d$, the polymer is in weak disorder. For all $\beta$ large enough, and all $h \in \R^d$, the polymer is in strong disorder.
\label{prop:gaussian random walk polymer in weak disorder for all h}
\end{prop}
A similar result holds for the high and low temperature regimes. This result does not tell us what happens for intermediate $\beta$, and since the proof of Proposition \ref{prop:gaussian random walk polymer in weak disorder for all h} shows that $\entropy(q(h))$ is periodic in $h$, it is conceivable that the polymer fluctuates between weak and strong disorder as $h$ is varied. We discuss this in \Secref{sec:strong disorder}.

Figure \ref{fig:gpl for different values of beta} shows the results of numerical experiments to compute $\gpl{h}$ in $d = 3$. The weights are chosen to be iid $\Uniform[0,1]$, and $p$ is the kernel of the standard nearest-neighbor random walk. Results are shown for $n=500$ steps, and $\beta \in \{1,3,5\}$. We averaged $\log Z_n$ over $100$ samples to estimate $\gpp{x}$ on a $[-n,n]^3$ grid, and used the Legendre transform in \eqref{eq:gpl in terms of gpp} to compute $\gpl{h}$. The GPU accelerated python code may be found \href{https://github.com/arjunkc/busemann-code}{here}\footnote{\href{https://github.com/arjunkc/busemann-code}{https://github.com/arjunkc/busemann-code}}. The simulation can be run for larger $n$ with more RAM, more patience, and/or cleverer code. 
\begin{figure}[ht]
\includegraphics[width=4.5cm]{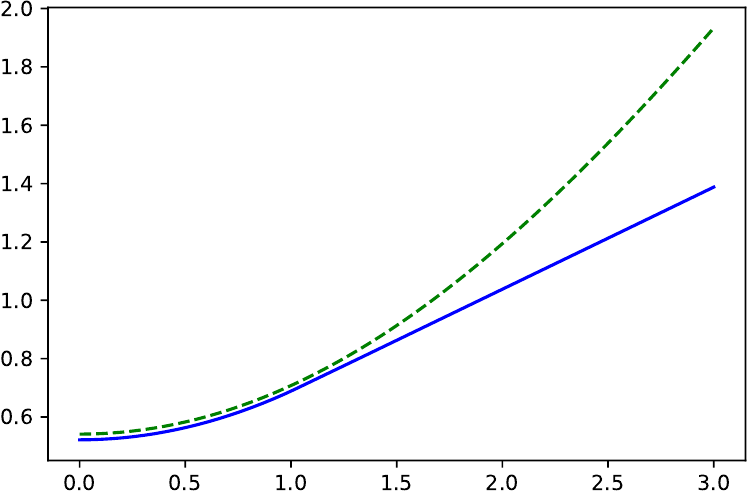}
\includegraphics[width=4.5cm]{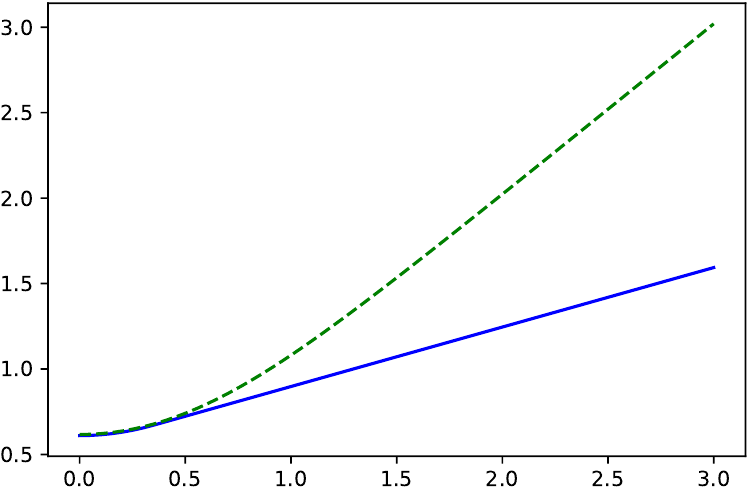}
\includegraphics[width=4.5cm]{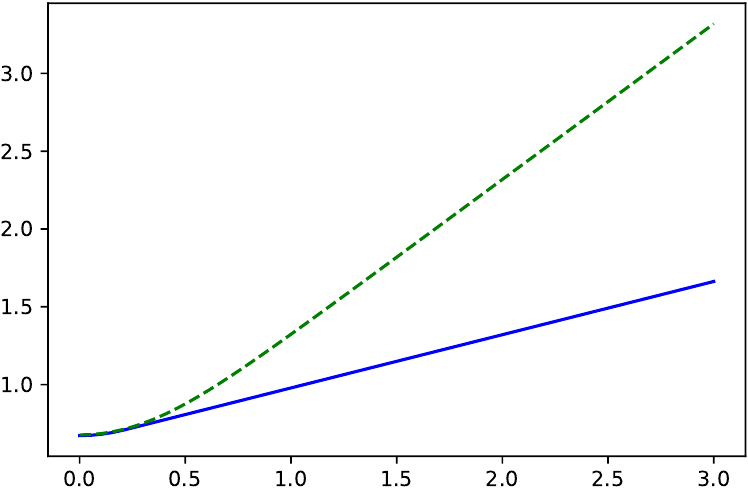}
\caption{
$\gpl{h}$ is in blue (solid), and the RHS of~\eqref{eq:annealed bound for gpl} is in green (dotted), each as a function of $h = t e_1$ 
From left to right, $\beta$ takes the values $1,3,\AND 5$.}
\label{fig:gpl for different values of beta}
\end{figure}
In these simulations, we can see a high temperature regime for $h$ near $0$ where $\gpl{h}$ coincides with the annealed bound. Then, the free energy smoothly transitions from the high-temperature regime to the low-temperature regime as $\gpl{h}$ becomes strictly smaller than the annealed bound. While the $\beta = 3$ and $\beta=5$ graphs appear similar, it seems clear that the phase transition happens at smaller $h$ for larger $\beta$. Our next few results indicate why this is true.
\begin{thm}
   The function $\beta \mapsto \log Z_n(\beta,h/\beta)/n - \lambda(\beta)$ is a nonincreasing function, and hence there is a critical $\beta_c(h)$ (resp. $\overline{\beta}_c(h)$) such that the polymer is at low temperature (resp. strong disorder) for $(\beta,h/\beta)$ when $\beta > \beta_c(h)$, and at high temperature (resp. weak disorder) for $\beta < \beta_c(h)$. 
   \label{thm:monotonicity}
\end{thm}
\notes{A similar statement applies with strong disorder replacing low temperature.}
In particular, this result shows that if $(\beta_0,h)$ is in the low temperature regime, then $(\beta,h \beta_0/\beta)$ is also in low temperature for $\beta > \beta_0$, and explains why low temperature appears in Figure \ref{fig:gpl for different values of beta} at smaller $h$ for larger $\beta$. This result is proved in Section \ref{sec:monotonicity} using ideas from the classical monotonicity result for $F(\beta) - \logmgf(\beta)$ \citep{MR2271480}.  

\notes{
\begin{rem}
   Since the effect of the field is to modify the underlying random walk \eqref{eq:transition probability in direction h}, the classical results like Theorem \ref{thm:full characterization of weak disorder bolthausen to comets et al} and Proposition \ref{prop:condition on gamma 2 for strong disorder} do apply, at least in the case where $p$ has finite-range. This implies that when $h$ is fixed, there is a $\beta_c(h)$ such that $(\beta,h)$ is in weak disorder when $\beta < \beta_c(h)$ and in strong disorder when $\beta > \beta_c(h)$. Theorem \ref{thm:simultaneosly in weak and strong disorder} implies that with $\beta$ fixed, there is a phase transition in the $h$ parameter which, to us, is quite surprising. In particular, this implies that $\liminf_{|h| \to \infty} \beta_c(h) = 0$, where $|\cdot|$ is the Euclidean norm on $\R^d$.
\end{rem}
Basically this remark is pretty wrong: we had trouble defining $\beta_c(h)$ at all since we only have monotonicy along hyperbolas $\beta \to (\beta,h/\beta)$. For this to work, we would have to change coordinates, and then we would have a bonafide phase transition. Moreover, people complained a lot about the phase transition not being surprising, so we just decided to remove it.}

Theorem \ref{thm:simultaneosly in weak and strong disorder} is proved in two different sections. Section \ref{sec:weak disorder} covers the weak-disorder portion of the theorem in Item \ref{item:weak disorder statement}. It follows a sketch in \citep[Exercise 9.1]{MR3444835} for the large deviation rate function for the polymer endpoint. However, step 2 of \citep[Exercise 9.1]{MR3444835} is nontrivial in the case of infinite range walks, and it deserves a proof that we could not find elsewhere in the literature. Our proof innovates on the ideas present in the proof of the standard local central limit theorem. A form of the estimate required in Item \ref{item:weak disorder statement} of Theorem \ref{thm:simultaneosly in weak and strong disorder} also appears in \citep[Lemma 5.3]{rassoul-agha_quenched_2012}, but again, the assumption $|\supp(p)| < \infty$ appears in an important way in their proof. A similar estimate also appears in \cite[Corollary 1.9]{junk_local_2023}, where again a drift is added to the underlying simple symmetric random walk. They prove that if $p^*(\beta) > 1 + 2/d$ (which is known to be true in the $L^2$ region), weak-disorder is preserved under the addition of small drifts. Our method removes the restriction of $|\supp(p)| < \infty$ in the $L^2$ region.

The strong disorder portion of Theorem \ref{thm:simultaneosly in weak and strong disorder} is covered in Section \ref{sec:strong disorder}. We use the technique of estimating fractional moments for the partition function originally used by \citet{MR0431355} for the multiplicative martingale, and its proof takes up the majority of Section \ref{sec:strong disorder}. \notes{I did not mention Proposition 2.4 in \citep{MR3672833}.}

One of the main consequences of the weak-strong disorder phase transition is that the Gibbs measure on the endpoint on the path changes from Gaussian to localized behavior. We present our generalization of these classical results next: 
\begin{enumerate}[1)]
\item a CLT in the $L^2$ region of $\weak_\beta$; i.e., in the set $\{ h \colon \mgf_2(\beta) < \pi(q(h))^{-1} \}$, and
\item localization of the endpoint for $h \in \strong_\beta$.
\end{enumerate}
The classical CLT for the polymer endpoint under the Gibbs measure is originally due to Bolthausen \citep{MR1006293} when the weights are $\Bernoulli(1/2)$. This was subsequently improved by \citet[eq.\,(2.5)]{MR2073332} using the version of the second moment computation for general weights by \citet{MR1413246}. They are both proved under the assumption of $|\supp(p)| < \infty$. Our generalization of these results involves an additional sub-Gaussianity assumption on $p$ to handle the case when $|\supp(p)| = +\infty$. It is of some interest to weaken this assumption.

For any positive-definite matrix $\Sigma$, we write $\gamma_\Sigma$ for the Gaussian measure with density
\begin{equation}
    \gamma_\Sigma(x) := \frac{1}{(2\pi )^{d/2} (\det \Sigma)^{1/2}} \exp \left( - \frac12 x \cdot \Sigma^{-1} x \right),
    \label{eq:gaussian measure with covariance matrix}
\end{equation}
where $x \cdot y$ is the usual dot product on $\R^d$.
\begin{thm}
    Let $p$ be a random walk kernel whose step distribution is sub-Gaussian; i.e., $\exists C > 0$ s.t.
    \begin{equation}
        E_p[e^{t \cdot |X_1|}] \leq \exp \left( C |t|^2/2 \right),
        \label{eq:subgaussianity assumption on p}
    \end{equation}
    If $\mgf_2(\beta) < \pi(q(h))^{-1}$, then $\forall f \in C(\R^d)$ with at most polynomial growth at $\pm \infty$, we have
    \begin{equation}
        \lim_{n \to \infty} \bigg\la f \left( \frac{X_n - n m_{q(h)}}{\sqrt{n}} \right) \bigg\ra_{n,\beta,q(h)}
        = \int f \left( x \right) d \gamma_{\Sigma(q(h))}(x) \quad \Prob-\almostsurely
        \label{eq:gaussian limit for polynomial functions of endpoint}
    \end{equation}
    where $m_{q(h)} = E_{q(h)}[X_1]$, and $\Sigma_{q(h)}$ is the covariance matrix of one step of the random walk given by
    \begin{equation*}
        \Sigma_{q(h)}(i,j) = E_{q(h)} \left[ \left( (X_1 - m_{q(h)}) \cdot e_i   \right) \left( (X_1 - m_{q(h)}) \cdot e_j \right) \right].
    \end{equation*}
    \label{thm:clt for the endpoint in weak disorder}
\end{thm}
\begin{rem}
    Choosing $f(x) = x$, we see that $\lim_{n \to \infty} \big\la \frac{X_n}{n} \big\ra_{n,\beta,q(h)} = E_{q(h)}[X_1] ~\almostsurely$. That is, the endpoint of the polymer is aligned with the drift of the underlying random walk induced by the field. If $p$ has finite range, Theorem \ref{thm:simultaneosly in weak and strong disorder} shows that a low temperature regime exists. This means that the polymer endpoint will be misaligned with the underlying random walk on a nontrivial subset of $h$ in the low temperature regime. 

    Choosing $f(x) = |x|^2$ shows the almost sure diffusive scaling of the endpoint location under the polymer measure.
    \label{rem:misalignment with underlying drift}
\end{rem}
\begin{rem}
    While we have proved a CLT for the endpoint in the $L^2$ region of $\weak_\beta$, it seems clear that one ought to be able to prove a Brownian invariance principle for the entire polymer path in weak disorder, \`a la \citet{MR2271480}. 
\end{rem}
Recall the definition of localization in \eqref{eq:localization definition}. It follows from the proofs of Theorems 5.1 and 5.4 in \citep{MR3444835}, that 
\begin{thm}
If $h \in \strong_{\beta}$, then $\sum_t J_t=\infty$. If, moreover, the polymer is in the low temperature regime, then the polymer endpoint is localized.
\end{thm}

In this paper, we have started the study of the full-phase diagram of the polymer model. We highlight a couple of interesting questions:
\begin{enumerate}
    \item Can one prove that there is no re-entrant phase transition in $h$, as Figure \ref{fig:gpl for different values of beta} suggests? The classical FKG proof of monotonicity in $\beta$ for $F(\beta) - \logmgf(\beta)$ in \citep{MR2271480} does not appear to work.
    \item Can one define and describe the boundary between $\weak_\beta$ and $\strong_\beta$?
\end{enumerate}

\subsection{Acknowledgements}
The authors would like to thank H.\,Lacoin for a valuable discussion, and F.\,Rassoul-Agha, T.\,Sepp\"al\"ainen and A.\,Yilmaz for comments on a first draft. 

A.\,Krishnan was partially supported by a Simons Collaboration grant 638966. S.\,Mkrtchyan was partially supported by a Simons Collaboration grant 422190. S.\,Neville was partially supported by a NSF RTG 1840234. This research was supported in part through computational resources and services provided by Advanced Research Computing at the University of Michigan, Ann Arbor and by the Center for Integrated Research Computing at the University of Rochester. 

\section{Existence of the free-energy}
\label{sec:shape theorem}

In this section, we prove Proposition \ref{prop:existence and convexity of gpl}. When $p$ is finite-range, there is a ``point-to-point'' analog in \citep[Theorem 2.2]{rassoul-agha_quenched_2013}; this and the duality relationship \eqref{eq:gpl in terms of gpp} imply Proposition \ref{prop:existence and convexity of gpl}. We establish the analogous theorem for the point-to-level free energy simultaneously for all temperatures $\beta \in (0,\infty)$ and external fields $h \in \R^d$.

\begin{proof}
    For each $(\beta,h) \in (0,\infty)\times\Q^d$, the existence of the limit follows directly from the proof of Proposition 2.5 in \citep{MR1996276}, which goes through without modification assuming \eqref{eq:assumptions on both mgfs}. Let 
    $$
    g_n(\beta,h) = \frac{ \log E_p  \left[ \exp\left(  \beta ( H(\vx_n) + h \cdot X_n)  \right) | X_0 = 0 \right] }{ \beta n}.
    $$
    We assume in the following that $X_0 = 0$ to avoid repeating the conditioning in the expectation. Then, it is easy to check that
    \begin{align}
        \nabla_h g_n(\beta,h) 
    & = n^{-1} \frac{ E_p \left[ X_n \exp\left(  \beta ( H(\vx_n) + h \cdot X_n)  \right)  \right] }{ E_p \left[ \exp\left(  \beta ( H(\vx_n) + h \cdot X_n)  \right)  \right] } \label{eq:numerator and denominator in shape theorem} \\
    & = n^{-1} \la X_n \ra_{n,\beta,q(h)}.\nonumber
    \end{align}
    Similarly, 
$\partial_\beta g_n(\beta,h) = n^{-1} \la H(\vx_n)+h\cdot X_n \ra_{n,\beta,q(h)}$.

    \begin{claim}
    \label{claim: derivative of free energy bounded}
     For any fixed $K>0,0<\beta_1<\beta_2$ and for all $|h| \leq K, \beta\in[\beta_1 , \beta_2]$, there is a constant $C$ such that eventually $\Prob$-almost-surely
     \begin{enumerate}
         \item \label{item:derivative of logZn in h bounded} $|\grad_h g_n(\beta,h)| \leq C$, and
         \item \label{item:derivative of logZn in beta bounded} $|\partial_\beta g_n(\beta,h)| \leq C$.
     \end{enumerate}
    \end{claim}

Note that we can write  $\partial_\beta g_n(\beta,h) = n^{-1} \la H(\vx_n)\ra_{n,\beta,q(h)}+h\cdot \grad_h g_n(\beta,h)$, so proving the Claim is equivalent to proving item \ref{item:derivative of logZn in h bounded} and 
     \begin{enumerate}
         \item[(3)] $|n^{-1} \la H(\vx_n)\ra_{n,\beta,q(h)}| \leq C$.
         \label{item:Gibbs expectation of H is bounded}
     \end{enumerate}

Claim \ref{claim: derivative of free energy bounded} shows that $g_n$ is eventually uniformly Lipschitz on $\{(\beta,h) \colon |h| \leq K,\beta_1\leq\beta\leq\beta_2\}$. Since $K,\beta_1$ and $\beta_2$ are arbitrary and $g_n(\beta,h)$ converges on $(\Q\cap(0,\infty)) \times \Q^d$, the Arzela-Ascoli theorem shows that $g_n(\beta,h)$ converges to a $\gpl{h}$ that is continuous on $(0,\infty) \times \R^d$.

Taking the second partial derivatives with respect to $h$, we see that 
\begin{equation*}
\beta^{-1} n\partial_{h_i}\partial_{h_j} g_n(\beta, h) = \la X_n \cdot e_i X_n \cdot e_j \ra_{n,\beta,q(h)} - \la X_n \cdot e_i \ra_{n,\beta,q(h)}\la X_n \cdot e_j \ra_{n,\beta,q(h)},
\end{equation*}
which is a covariance, and thus the Hessian is positive definite. Therefore,  $g_n(\beta,h)$ is convex in $h$, and thus, so is $\gpl{h}$. Similarly, the second partial derivative of $g_n(\beta,h)$ with respect to $\beta$ is the variance of $H(\vx_n)+h\cdot X_n$, and hence $g_n(\beta,h)$ and $\gpl{h}$ are convex in $\beta$ as well. The Hessian of $g_n(\beta,h/\beta)$ is again a covariance matrix, which implies the joint convexity of $\gpl{h/\beta}$ in $(\beta,h)$.

We first prove item \ref{item:derivative of logZn in h bounded} of the claim. When $p$ is finite-range, this is obvious; in the general case, we will use the fact that both $\mgf(s)$ and $\mgfrw_p(t)$ are finite on their respective domains, and thus, good large deviations estimates are available. 
        
Let the numerator and denominator of the fraction in \eqref{eq:numerator and denominator in shape theorem} be $\operatorname{NUM}$ and $\operatorname{DENOM}$ respectively. To prove item \ref{item:derivative of logZn in h bounded} we bound the $\operatorname{NUM}$ from above and $\operatorname{DENOM}$ from below on an event with high probability. Let
    \begin{equation*}
    \operatorname{END}(C_0) = \{ y \in \Z^d \colon |y - n E_p[X_1] | \leq C_0 n \}.
    \end{equation*}
We split up the expectation in the numerator into two terms,
    \begin{align}
        \operatorname{\operatorname{NUM}} 
        & = E_p \left[ X_n \exp\left(  \beta ( H(\vx_n) + h \cdot X_n)  \right) (1_{\operatorname{END}(C_0)} + 1_{\operatorname{END}(C_0)^c})  \right]  
        \label{eq:definition of numerator in upper bound of derivative log Zn}
    \end{align}
    and note that the second term is the one we must control since $X_n$ is well-behaved in the first. Call this term $T_2$. We have
    \begin{align*}
       |T_2| & \leq E_p \left[ |X_n| \exp\left(  \beta ( H(\vx_n) + h \cdot X_n)  \right)  1_{\operatorname{END}(C_0)^c}\right]  \\
       & \leq E_p \left[  \exp\left(  \beta  H(\vx_n) +  (\beta_2 K + 1) |X_n|   \right) 1_{\operatorname{END}(C_0)^c}  \right]\\
       & =: \hat{T}_2
    \end{align*} 
    where we have used H\"older to bound the $h \cdot X_n$ term, the condition $|h| \leq K$, and the inequality $x<e^x$ to absorb the term $|X_n|$ into the exponential. Now $\hat{T}_2$ is independent of $h$ and much easier to control. Taking expectation, we get
    \begin{align*}
        \E\left[\sup_{\beta\in[\beta_1,\beta_2]} \hat{T}_2\right] 
        & =  E_p \left[ \E\left[\sup_{\beta\in[\beta_1,\beta_2]}\exp(\beta H(\vx_n))\right] \exp\left(   (\beta_2 K + 1) |X_n|   \right) 1_{\operatorname{END}(C_0)^c}  \right].
    \end{align*} 
    Splitting the inner expectation into the events $H(\vx_n)>0$ and $H(\vx_n)<0$, we get 
    \begin{equation*}
    \E\left[\sup_{\beta\in[\beta_1,\beta_2]}\exp(\beta H(\vx_n))\right]\leq \mgf(\beta_1)^n+\mgf(\beta_2)^n .
    \end{equation*}
    Taking this term out, and using the Cauchy-Schwartz inequality for the remaining $E_p$-expectation, we obtain 
    \begin{align}
        \E\left[\sup_{\beta\in[\beta_1,\beta_2]} \hat{T}_2\right] 
        & \leq  (\mgf(\beta_1)^n+\mgf(\beta_2)^n) \overline{M}_p(d (\beta_2 K + 1))^{n/2} P_p( \operatorname{END}(C_0)^c )^{1/2}. \label{eq:bound needed for numerator in shape theorem}
    \end{align} 
    In the previous line, $\overline{M}_p$ is given by
    \begin{align*}
        E_p [ \exp( C |X_1|) ]  
        & \leq E_p [ \exp( C |X_1|_1 ) ]\\
        & \leq \prod_{j=1}^d E_p \left[ \exp( C d |X_1 \cdot e_j| ) \right]^{1/d}\\
        & \leq \max_{i=1,\ldots,d} (M_p(C d e_i) + M_p(-C d e_i)) \\
        & : = \overline{M}_p(Cd),
    \end{align*} 
    where $|\cdot|_1$ is the $\ell^1$ norm. Since $M_p(t) < \infty$ is finite for all $t$, $X_n$ has a good rate function, and $P_p( \operatorname{END}(C_0)^c ) \leq \exp( - n (I(C_0) + o(1) ))$ where $I(C_0)>0$ for all $C_0>0$ and $I(C_0) \to \infty$ as $C_0 \to \infty$ (see Cram\'er's theorem \citep{rassoul2015course}). Since $K$, $\beta_1$ and $\beta_2$ are fixed, the right-hand side of \eqref{eq:bound needed for numerator in shape theorem} can be bounded by $\exp( - C_1 n)$ where $C_1$ can be chosen to be arbitrarily large by choosing $C_0$ to be large. Using Chebyshev's inequality we obtain, 
\begin{equation}
    \Prob(|\hat{T}_2|>\exp(-C_1n/2))\leq \exp(-C_1n/2).
    \label{eq:exponential bound for THat2}
\end{equation}
    
    

    Next, we need a lower bound for the denominator in \eqref{eq:numerator and denominator in shape theorem}. It will be enough for our purposes to bound it by some exponentially decaying function in $n$, as long as the rate is not too small.

    As long as $E_p[ p(X_1)^u ]$ is finite in some interval around the origin, Cram\'er's theorem again gives us a good rate function $I_{1}$ such that $I_1(t) > 0$ for all $t \neq 0$ \cite[Exercise 3.8]{rassoul2015course}, and 
    \begin{equation*}
        \frac{1}{n}\log  P_p\left(  \bigg| \sum_{i=1} \left(\log  p (X_{i} - X_{i-1}) - E_p[ \log p(X_1) ] \right) \bigg|  \geq n t \right) \to - I_1(t).
    \end{equation*}
    $E_p[ p(X_1)^u ]$ can be shown to be finite for $u \in (-1,1)$ as follows: Since $\mgfrw_p$ is finite on its domain, Chebyshev's inequality gives $P_p( |X_1| \geq s) \leq \exp(-C_1 s) E_p[\exp(C |X_1|)]$. Then,
    \begin{equation}
    \label{eq:p(X_1)_finite_moment}
        E_p[ p(X_1)^u ] = \sum_{x \in \Z^d} p(x)^{1+u} \leq \sum_{k =0 }^{\infty} C_2 k^{d-1} e^{- C_1(1 + u) k} < \infty.
    \end{equation}
    Thus, there is a set $Y_1$ of random walk paths such that $P_p( Y_1^c) \leq \exp(-n I_1(t)/2)$ and
    \begin{equation*}
        \exp( n ( E_p[ \log p(X_1) ] - t) ) \leq P_p(\vx_n) \leq \exp( n ( E_p[ \log p(X_1) ] + t) ) \quad \forall \vx_n \in Y_1.
    \end{equation*}
    Choosing $t =  -E_p[ \log p(X_1) ]/2$, it follows that $|Y_1| \leq  \exp( n ( - 3 E_p[ \log p(X_1) ]/2) )$. 

    We can control the weight on the exponentially many paths in $Y_1$ using large deviation bounds. Since $\mgf(\beta)<\infty$, there exists a good rate function $I_2$ for iid sums of weights. So, for any $C>0$, there exists $c>0$ such that 
    \begin{equation}
        \Prob \left ( \bigcup_{\vx_n \in Y_1} \left\{ | H(\vx_n) - n\E[\w_0]|> cn \right\} \right) 
        \leq |Y_1| \exp( - n C). \label{eq:event on which weights are lower bounded in existence proof}
    \end{equation}
By choosing $c$ to be large enough, we can get $C>-E_p[ \log p(X_1) ] + t$, and thus the probability of the event in \eqref{eq:event on which weights are lower bounded in existence proof} decays exponentially in $n$.

Since $\operatorname{DENOM}$ is the $P_p$-expectation of a positive random variable, we can restrict to a smaller set of paths $Y_1 \cap \operatorname{END}(C_2)$. On the complement of the event in \eqref{eq:event on which weights are lower bounded in existence proof} we obtain
    \begin{align*}
        \operatorname{DENOM} 
        & = E_p \left[ \exp( \beta H(\vx_n) + \beta X_n \cdot h ) \right] \\
        & \geq E_p\left(\exp( n \beta ( \E[\w_0] - c)) \exp(- n \beta K (C_2+|E_p[X_1]|)) 1_{Y_1 \cap \operatorname{END}(C_2)} \right)\nonumber \\
        & \geq \frac 12 \exp( n \beta ( \E[\w_0] - c- K (C_2+|E_p[X_1]|)),     \end{align*}
    where we have again used the H\"older inequality to control $h \cdot X_n$ for $X_n \in \operatorname{END}(C_2)$, and $P_p(Y_1 \cap \operatorname{END}(C_2))\geq \frac 12$, which holds since both $Y_1$ and $\operatorname{END}(C_2)$ have complements with exponentially small probabilities. 
  Note that here $C_2$ is an arbitrary positive constant and $c>0$ depends on the constant $C$ in \eqref{eq:event on which weights are lower bounded in existence proof}, so there exists a constant $C_3>0$, independent of $C_0$ and $C_1$, such that with probability close to $1$, we have
    \begin{align}
        \operatorname{DENOM} 
        & \geq  \exp( -n \beta C_3). \label{eq:lower bound on denominator in existence proof}
    \end{align}

Combining \eqref{eq:exponential bound for THat2} and \eqref{eq:lower bound on denominator in existence proof}, on the complement of the union of the events in \eqref{eq:exponential bound for THat2} and \eqref{eq:event on which weights are lower bounded in existence proof} we obtain 
\begin{align*}
\left|\frac{\operatorname{NUM}}{\operatorname{DENOM}}\right|&\leq (C_0+|E_p[X_1]|)n+\frac{|\hat{T}_2|}{\operatorname{DENOM}}\\&\leq (C_0+|E_p[X_1]|)n+\exp(-(C_1/2-\beta C_3)n).
\end{align*}
Dividing by $n$ and using the Borel-Cantelli lemma, we get $|\nabla_h g_n(\beta,h)| \leq 2C_0$ eventually almost-surely.

    The proof of item \ref{item:Gibbs expectation of H is bounded} of the claim is similar. The lower bound for the denominator is identical. As in \eqref{eq:definition of numerator in upper bound of derivative log Zn}, we write the numerator of $\la H(\vx_n)\ra_{n,\beta,q(h)}$ as
    \begin{multline}
        E_p \left[ H(\vx_n) \exp\left(  \beta ( H(\vx_n) + h \cdot X_n)  \right) 1_{Y_1}  \right]\\+E_p \left[ H(\vx_n) \exp\left(  \beta ( H(\vx_n) + h \cdot X_n)  \right)  1_{Y_1^c}  \right] . 
        \label{eq:definition of numerator in upper bound of derivative log Zn in beta}
    \end{multline}
Using the inequality $|H(\vx_n)|\leq e^{H(\vx_n)} + e^{-H(\vx_n)}$ and $|h|\leq K$, the second term in \eqref{eq:definition of numerator in upper bound of derivative log Zn in beta} can be bounded from above by
\begin{equation*}
\hat{T}_3= E_p \left[  \exp\left(  (\beta+1)  H(\vx_n) +  \beta_2 K  |X_n|   \right)1_{Y_1^c}  \right] +
E_p \left[  \exp\left(  (\beta-1)  H(\vx_n) +  \beta_2 K  |X_n|   \right)1_{Y_1^c}  \right].
\end{equation*}
As in \eqref{eq:bound needed for numerator in shape theorem} we obtain 
    \begin{align*}
        \E\left[\sup_{\beta\in[\beta_1,\beta_2]} \hat{T}_3\right] 
        & \leq  4 \left( \max_{\beta \in [\beta_1-1,\beta_2+1]\}} \mgf(\beta)^n \right) \overline{M}_p(d \beta_2 K )^{n/2} P_p(Y_1^c )^{1/2}, \label{eq:bound needed for numerator in shape theorem}
    \end{align*} 
which is exponentially small in $n$ if the parameter $t$ in the definition of $Y_1$ is large enough. Again, as in \eqref{eq:exponential bound for THat2}, Chebyshev's inequality implies that $|\hat{T}_3|$ is exponentially small with probability exponentially close to $1$.

To bound the first term in \eqref{eq:definition of numerator in upper bound of derivative log Zn in beta} with high probability, note that on the complement of the event in \eqref{eq:event on which weights are lower bounded in existence proof}, for all $\vx_n\in Y_1$ we have $|H(\vx_n)|\leq(c+|\E[\w_0]|)n$ which implies that the first term in \eqref{eq:definition of numerator in upper bound of derivative log Zn in beta} is bounded by $(c+|\E[\w_0]|)n\operatorname{DENOM}$. Combining the bounds for both terms in \eqref{eq:definition of numerator in upper bound of derivative log Zn in beta} gives the necessary bound for the numerator in item \ref{item:Gibbs expectation of H is bounded}.

    This completes the proof of the claim and the proposition, establishing existence and convexity of $\gpl{h}$. 
\end{proof}

\section{Weak disorder}
\label{sec:weak disorder}
In this section, we compute $\pi(q)$, the intersection probability of two independent copies of the random walk with kernel $q$, so that we can apply the classical second-moment condition \eqref{eq:L2 regime defining condition} to determine if the walk is in weak disorder. This is equivalent to determining the return probability to the origin of the difference of the two independent random walks; let $p$ be the transition kernel of the difference. Then, $\origin \in \supp(p)$, and it is symmetric about the origin; i.e., $x \in \supp(p)$ implies $-x \in \supp(p)$. This implies that the range of the random walk ---nonnegative integer linear combinations of $\supp(p)$--- is the integer span of $\supp(p)$; i.e., it is a lattice. It is a well-known fact that lattices have a basis, and these can be put into a matrix $A$. For example, when the two random walks are standard nearest-neighbor random walks on $\Z^2$, then $\origin \in \supp(p)$, and the integer span of $\supp(p)$ is the lattice $A \Z^d$ where $A = \begin{pmatrix} 1 & 1 \\ 1 & -1 \end{pmatrix}$. 
\begin{define}
    We call the random walk on $\Z^d$ with symmetric kernel $p$ truly $d$-dimensional if the sublattice  $A \mathbb{Z}^d$ obtained by the integer span of the support of $p$ is in fact $d$-dimensional. That is, there is an invertible matrix $A$ with coefficients in $\Z$ such that 
    \begin{equation}
        \Z - \operatorname{span}(\supp(p)) = \left\{ \sum_{i=1}^k a_i z_i \colon a_i \in \supp(p), z_i \in \Z \right\} = A \Z^d.    
      \label{def:truly d dimensional if support is a sublattice of Zd}
\end{equation}
\end{define}
\notes{Do we need a reference for the lattice basis here?}
 
To compute the return probability to the origin, we use the proof of the following generalization of the standard local central-limit theorem. 
\begin{thm}[c.f.\,Theorem 3.5.3\citep{Durrett_2019}]
   Let $p$ be symmetric, truly $d$-dimensional, $\origin \in \supp(p)$, and generate the sublattice $A \Z^d$. Let $m_p = E_p[X_1]$ and $\Sigma$ be the covariance matrix of one step of the random walk under $p$. Then
   \begin{equation*}
       \lim_{n \to \infty} \sup_{x \in A \Z^d} n^{d/2} \left| P_p(S_n = x) - \frac{\det(\Sigma^{-1/2}A)}{(2\pi n)^{d/2}} \exp \left( -\frac{\left( (x - m_p n),\Sigma (x - m_p n)  \right)}{2n}   \right) \right| = 0.
   \end{equation*}
   \label{thm:our version of the local clt}
\end{thm}
We do not prove \ref{thm:our version of the local clt} since it is standard, but will follow and modify its proof for the next theorem.
\begin{thm}
    Suppose $d \geq 3$, and $\beta > 0$. Assume that the condition for weak disorder in \eqref{eq:L2 regime defining condition} is satisfied, and that $p$ is truly $d$-dimensional. Then, there is a $C_0 > 0$ such that for all $\norm{h}{} < C_0$, we have $\mgf_2(\beta) < 1/\pi(q(h))$, where $q(h)$ is defined in \eqref{eq:transition probability in direction h}; i.e., the polymer is in the weak disorder regime under the external field $h$.
    \label{thm:weak disorder holds for small h}
\end{thm}

\begin{proof}[Proof of Theorem \ref{thm:weak disorder holds for small h}]
Since the left-hand side of the inequality $\mgf_2(\beta) < 1/\pi(q(h))$ is independent of $h$, and we are assuming that the inequality \eqref{eq:L2 regime defining condition} holds when $h=\origin$, it is enough to show that $\pi(q(h))$ is continuous (in $h$) at $h=\origin$. Recall that we have
\begin{align*}
\pi(q(h))=&P_{q(h)}(S_n=S_n'\text{ for some }n > 0),
\end{align*}
which can be interpreted to be the return probability of the origin of the random walk $T_n=S_n-S_n'$, where $S_n$ and $S_n'$ are two independent copies of the same random walk. Using the Markov property for this random walk, we have 
\begin{align}
R(q(h)) & := \sum_{n=1}^\infty P_{q(h)}(T_n=\origin)=E_{q(h)}[\text{\# of returns of }T_n] \nonumber \\
&=\sum_{k=1}^\infty k\pi^k(q(h))(1-\pi(q(h))) = \frac{\pi(q(h))}{1-\pi(q(h))}, \label{eq:sum defining expected number of returns of random walk}
\end{align}
where the RHS is to be interpreted as $\infty$ if $\pi(q(h)) = 1$. In $d \geq 3$, since $\pi(q(\origin)) < 1$, the continuity of $\pi(q(h))$ at $h=\origin$ is equivalent to the continuity of $R(q(h))$ at $h=\origin$. Later on in the proof, we show that the tails of the sum defining $R(q(h))$ in \eqref{eq:sum defining expected number of returns of random walk} decay uniformly for $|h|$ small enough; i.e.,  given $\varepsilon>0$, there exists $C_1>0$ and $n_\varepsilon$ such that for all $|h|<C_1$ and $n>n_\varepsilon$, we have $\sum_{n>n_\varepsilon}P_{q(h)}(T_n=\bar{x})<\varepsilon/3$. Thus, 
\begin{align}
    |R(q(h)) - R(q(\origin))|
    & = \left|\sum_{n}P_{q(h)}(T_n=\origin)-P_{q(\origin)}(T_{n}=\origin)\right| \nonumber \\
    & < 2\varepsilon/3+\left|\sum_{n<n_\varepsilon}P_{q(h)}(T_n=\origin)-P_{q(\origin)}(T_n=\origin)\right|. \label{eq:inequality in estimate for expected number of intersectionso of RW}
\end{align}
From the definition \eqref{eq:transition probability in direction h} of the transition kernel $q(h)$, 
it is easy to see that $E_{q(h)}[e^{i u T_n}]$ is a continuous function of $h$: Let $B = \{\pm 1\}^d$. If $|h|_{\infty} \leq C$, then $h \cdot x \leq C (\operatorname{sgn}(x_1),\ldots,\operatorname{sgn}(x_n)) \cdot x$. Therefore, 
\begin{align*}
    |E_{q(h)}[e^{i u T_n}]|
    & \leq \mgfrw_p(\beta h)^{-1} \sum_{x \in \Z^d} p(x) e^{\beta h \cdot x} \\ 
    & \leq \mgfrw_p(\beta h)^{-1} \sum_{x \in \Z^d} p(x) \sum_{y \in B} e^{C \beta y \cdot x} = \mgfrw_p(\beta h)^{-1} \sum_{y \in B} \mgfrw_p(C \beta y)  
\end{align*}
If $h_k \to h$ and $|h_k|_\infty \leq C$, the dominated convergence theorem implies $E_{q(h_k)}[e^{i u T_n}] \to E_{q(h)}[e^{i u T_n}]$. From the Levy continuity theorem, it follows that $P_{q(h)}(T_n = \origin)$ is continuous in $h$, and thus, the sum in \eqref{eq:inequality in estimate for expected number of intersectionso of RW} converges to zero when $h\to \origin$. So $\exists C_2>0$ such that $\forall |h|<C_2$, we have $|\sum_{n<n_\varepsilon}P_{q(h)}(T_n=\origin)-P_{q(\origin)}(T_n=\origin)|<\varepsilon/3$. Setting $C_0=\min(C_1,C_2)$ we obtain the desired continuity of $R(q(h))$.

Next, we establish the uniform in $h$ bound on the tail of the sum in \eqref{eq:sum defining expected number of returns of random walk}, by using ideas from the proof of the local limit theorem (see e.g. \cite{MR2722836}). Consider the $\mathbb{Z}$-span of the support of the random walk $T_{1}$. By assumption, it forms a sublattice of $\mathbb{Z}^d$ which has dimension $d$. It follows that there exists an invertible matrix $A$ with coefficients in $\mathbb{Z}$ such that the sublattice in question is equal to $A\mathbb{Z}^d$. Note that different choices of bases for the sub-lattice correspond to different matrices $A$, but this will not play a major role in the proof. 

We have the identity 
$$
1_{\{T_1=\bar{x}\}}=1_{\{A^{-1}T_1=A^{-1}\bar{x}\}}=\frac{1}{(2\pi)^d}\int_{-\pi}^{\pi}\cdots \int_{-\pi}^{\pi}e^{iA^{-1}T_1\cdot u}e^{-iA^{-1}x\cdot u}du.
$$ Taking expectation, we get
\begin{equation*}
    P_{q(h)}(T_1=\bar{x})=\frac 1{(2\pi)^d}\int_{-\pi }^{\pi }\dots\int_{-\pi}^{\pi }\phi_{q(h)}(u)e^{-iA^{-1}\bar{x}\cdot u}du,
\end{equation*}
where $\phi_{q(h)}$ is the characteristic function of $A^{-1}T_1$. More generally, for the $n$\textsuperscript{th} step of the random walk, after a change of coordinates, we can write 
\begin{equation}
    P_{q(h)}(T_n=\bar{x})=\frac 1{n^{d/2}}\frac 1{(2\pi)^d}\int_{-\pi n^{1/2}}^{\pi n^{1/2}}\dots\int_{-\pi n^{1/2}}^{\pi n^{1/2}}\phi_{q(h)}^n(u/\sqrt n)e^{-i\frac{A^{-1}\bar{x}\cdot u}{\sqrt n}}du.
    \label{eq:probability of hitting the origin after n steps in terms of char function}
\end{equation}
Setting $\bar{x}=\origin$, we get
$$
R(q(h))=\sum_{n=1}^\infty \frac 1{n^{d/2}}\frac 1{(2\pi)^d}\int_{-\pi n^{1/2}}^{\pi n^{1/2}}\dots\int_{-\pi n^{1/2}}^{\pi n^{1/2}}\phi_{q(h)}^n(u/\sqrt n)du.
$$

Since $\phi_{q(h)}$ is a characteristic function, we have $|\phi_{q(h)}(u)|\leq 1$ for all $u\in\mathbb{R}^d$. However, in the region $0<|u_i|\leq \pi$, we must have the strict inequality $|\phi_{q(h)}(u)|<1$. If not, for some $u$ in this region, we would have $|\phi_{q(h)}(u)|=1$. This implies that $e^{iA^{-1}T_1\cdot u}$ is almost surely constant, so $A^{-1}T_1\cdot u\in c+2\pi\mathbb{Z}$ for all possible values of $T_1$ for some constant $c$. Since $T_1=S_1-S_{1}'$, we have that $\origin$ is in the support of $T_1$, and so, without loss of generality $c=0$, giving $A^{-1}T_1\cdot u\in 2\pi\mathbb{Z}$. Since the right-hand side is a lattice, we get the same containment by replacing $T_1$ by the $\mathbb{Z}$-span of its support, namely $A\mathbb{Z}^d$, giving us $A^{-1}A\mathbb{Z}^d\cdot u=\mathbb{Z}^d\cdot u\in 2\pi\mathbb{Z}$. This is impossible for any $u$ such that $0<|u_i|\leq \pi$. Thus, for any $\delta>0$, using the continuity of $\phi_{q(h)}(u)$ in $h$ and $u$, there exists a constant $c\in(0,1)$ such that in the region $|h| \leq C_1$ and $\delta<\max_i |u_i|\leq \pi$, we have $|\phi_{q(h)}(u)|\leq c$. 

Then, we can estimate the integral in \eqref{eq:probability of hitting the origin after n steps in terms of char function} as follows: 
\begin{align*}
&\left|\int_{[-\pi n^{1/2},\pi n^{1/2}]^d} \phi_{q(h)}^n(u/\sqrt n)du\right|
\\&\quad\leq 
\int_{[-\delta n^{1/2},\delta n^{1/2}]^d}\left|\phi_{q(h)}^n(u/\sqrt n)\right|du
+\int_{[-\pi n^{1/2},\pi n^{1/2}]^d\backslash[-\delta n^{1/2},\delta n^{1/2}]^d}\left|\phi_{q(h)}^n(u/\sqrt n)\right|du
\\&\quad\leq \int_{[-\delta n^{1/2},\delta n^{1/2}]^d}\left|\phi_{q(h)}^n(u/\sqrt n)\right|du+
c^n (2\pi)^dn^{d/2}.
\end{align*}

Using \citep[Theorem 3.3.20]{Durrett_2019}, we can write
$$|\phi_{q(h)}(u)|\leq 1-\frac 12E_{q(h)}|u\cdot T_1|^2+E_{q(h)}\min(|u\cdot T_1|^3,2|u\cdot T_1|^2).$$
If $u$ is sufficiently close to the origin, the minimum in the last expectation is taken by the cubic term, and we get 
$$|\phi_{q(h)}(u)|\leq 1-\frac 14E_{q(h)}|u\cdot T_1|^2\leq e^{-\frac 14E_{q(h)}|u\cdot T_1|^2},$$ 
which implies 
$$|\phi^n_{q(h)}(u/\sqrt{n})|\leq e^{-\frac 14E_{q(h)}|u\cdot T_1|^2} \leq e^{-C |u|^2}$$ 
since $q(h)$ is a continuous function of $h$, and $|h| < C_1$. Thus, the tail $n\geq m$ of the series for $R(q(h))$ can be estimated from above by
$$\sum_{n=m}^\infty\frac 1{n^{d/2}}\frac 1{(2\pi)^d}\left(\int_{[-\delta n^{1/2},\delta n^{1/2}]^d}e^{-C |u|^2}du+
c^n (2\pi)^dn^{d/2}\right),$$
which, if $m$ is large enough, is arbitrarily small uniformly in $h$, as needed.
\notes{Basically we use 
    $$
    E_p[ e^{\beta h \cdot (X_1 + X_1')} (u \cdot (X_1 - X_1'))^2 ] 
    $$
    is a continuous function in $h$.
}
\end{proof}

\subsection{Endpoint behavior in weak disorder}
\label{label:endpoint behavior in weak disorder}
In this section, we sketch the proof of Theorem \ref{thm:clt for the endpoint in weak disorder}. It follows the presentation in \citet[Theorem 3.4]{MR3444835} to a T, and so we will merely indicate where changes need to be made for a general underlying walk $p$. The best strategy to read this sketch is to read the version in \citep{MR3444835} first, and then return here to understand the changes. The only place where the transition probability $p$ matters is in the proof of \Propref{prop:growth rate of polynomial martingales in weak disorder} below.

Without loss of generality, we will assume that $m = E_p[X_1] = 0$ since we can always subtract the mean from the random walk and shift the lattice appropriately. Let $\phi_k(n,x)$ be a polynomial in $n$ and $x$ such that $\phi_k(n,X_n)$ is a martingale for the random walk filtration $\rwF_n = \sigma( X_0, \ldots, X_n)$ satisfying the growth condition
\begin{equation}
    \label{eq:growth condition on the polynomial martingales for the random walk}
    |\phi_k(n,x)| \leq C_0 + C_1 |x|^k + C_2 n^{k/2} \quad \forall n,x \in \Z_{\geq 0} \times \Z^d. 
\end{equation}

\begin{prop}[Proposition 3.2 in \citep{MR3444835} or Proposition 3.2.1 in \citep{MR2073332}]
   Suppose $d \geq 3$, $p$ is a random walk kernel whose step distribution is sub-Gaussian, $\mgf_2(\beta) < \pi(p)^{-1}$, and $\phi_k(n,X_n)$ is a polynomial martingale for the random walk filtration $\rwF_n$. Then,
   \begin{equation*}
       M_{n} := E_p\left[ \phi_k(n,X_n) \exp\left(  \beta H(\vx_n) - n \log \logmgf(\beta) \right) \right] 
   \end{equation*}
   is a martingale for the filtration $\weightsF_n =\sigma\left(\{\w_{t,x}\}:{t \leq n, x \in \Z^d}\right)$, and satisfies
   \begin{enumerate}
       \item $\exists \gamma \in [0,k/2)$ such that $\max_{0 \leq j \leq n} |M_{j}| = O(n^\gamma)$ almost surely,
       \item \label{it:second moment of martingale is finite} $\E[ M_{n}^2] = O(b_{n,k})$, where 
           $$
           b_{n,k} = 
           \begin{cases}
              1 & k < d/2 - 1 \\ 
              \log(n) & k = d/2 - 1\\
              n^{k-d/2+1} & k > d/2 -1
           \end{cases}
           $$
   \end{enumerate}
   \label{prop:growth rate of polynomial martingales in weak disorder}
\end{prop}
\begin{proof}[Proof of \Propref{prop:growth rate of polynomial martingales in weak disorder}]
   We follow the proof of \citep[Proposition 3.2.1]{MR2073332}, which is the corresponding statement in the case when $p$ is the simple random walk. In that proof the fact that $p$ is the simple random walk is only used to obtain the estimate (3.15) \citep[Lemma 3.2.2]{MR2073332}, which, in our notation, states 
\begin{equation}
\label{eq:k'th moment when independent walks intersect}
       E_p^{\otimes 2} \left[ |X_n |^{2k} 1_{X_n = X_n'} \right] 
       \leq C n^{k - d/2}.
\end{equation}

While \eqref{eq:k'th moment when independent walks intersect} is not true for general random walks $p$ with an exponential moment generating function, we show that it holds under the sub-Gaussianity assumption in \eqref{eq:subgaussianity assumption on p}. First, we note that if $p$ is sub-Gaussian, we have the tail estimate $P_p( |X_n| \geq t ) \leq e^{-t^2/(2 C n)}$ and the moment bound
   \begin{equation}
       E_p[ |X_n|^k ] = k \int_0^\infty t^{k-1} P_p( |X_n| \geq t) dt \leq C_k n^{k/2} \quad \forall n.
       \label{eq:gaussian moment estimate for random walks}
   \end{equation}
Using this estimate, we obtain
   \begin{align*}
       E_p^{\otimes 2} \left[ |X_n |^{2k} 1_{X_n = X_n'} \right] 
       & = \sum_{x \in \Z^d}  E_p \left[ |X_n |^{2k} 1_{X_n = x} \right] P_p \left[ X_n' = x \right]\\
       & \leq \frac{C}{n^{d/2}} \sum_{x \in \Z^d}  E_p \left[ |X_n |^{2k} 1_{X_n = x} \right] \\
       & = \frac{C}{n^{d/2}}  E_p \left[ |X_n |^{2k} \right] \\
       & \leq C n^{k - d/2},
   \end{align*}
   where we have used the local limit theorem (Theorem \ref{thm:our version of the local clt}) and \eqref{eq:gaussian moment estimate for random walks}. Note that in the above computations, $C$ is a constant that may change from line to line. Using this estimate, the rest of the argument is identical to \citep{MR2073332} or \citep{MR3444835} and proves Proposition \ref{prop:growth rate of polynomial martingales in weak disorder}.
\end{proof}

\begin{proof}[Proof of Theorem \ref{thm:clt for the endpoint in weak disorder}]
   Next, we construct the family of martingales $\phi_a(n,X_n)$ indexed by $a \in \Z_{\geq 0}^d$ satisfying \eqref{eq:growth condition on the polynomial martingales for the random walk}. Let $a=(a_j)_{j=1}^d$, and 
   $$
   |a|_1 = a_1 + \cdots + a_d, x^a = x_1^{a_1}\cdots x_d^{a_d}.
   $$
   We introduce
   \begin{equation}
       \phi_a(n,x) = \left( \frac{\partial}{\partial \theta} \right)^a \exp\left( \theta \cdot x - n \rho(\theta) \right) \bigg\vert_{\theta = \origin},
       \label{eq:definition of polynomial martingales in weak disorder clt proof}
   \end{equation}
   where $\rho(\theta) = \logmgfrw(\theta)$ is the log moment generating function of $X_1-X_0$, the first step of the random walk with kernel $p$. It is easy to check that $\phi_a(n,X_n)$ is a polynomial martingale that satisfies \eqref{eq:growth condition on the polynomial martingales for the random walk}. With this setup, the proof of Theorem \ref{thm:clt for the endpoint in weak disorder} is identical to \citep[Theorem 3.4]{MR3444835}. In our case, the function $\psi_a(n,x)$ needed in the proof is the Gaussian version of $\phi_a(n,x)$; i.e., $\rho(\theta)$ is replaced by $\hat\lambda(\theta) := (\theta,\Sigma\, \theta)/2$ in \eqref{eq:definition of polynomial martingales in weak disorder clt proof}.

   \notes{This makes some of the statements in \citep{MR3444835} clearer, at least for me, but I'm not sure it's worth including.
   It is enough to prove \Thmref{thm:clt for the endpoint in weak disorder} for monomial functions $f(x) = x^a$. The proof is by induction on $|a|_1$. 
   The base case of $|a|_1=0$ is, of course, trivial. 
   Fix any $a \neq \origin$. By the induction hypothesis, for any $b$ such that $|b|_1 < |a|_1$, we have
   \begin{equation*}
       \la \left( \frac{X_n}{\sqrt{n}} \right) \ra \to \int x^b d\gamma_{\Sigma}(x).
   \end{equation*}
   Our version of $(3.17)$ in \citep{MR2073332} is
   \begin{equation}
     \int \psi_a(1,x) d\gamma_\Sigma(x) = 0    
     \label{eq:our version of 3.17 in comets}
   \end{equation}
   To see this, let $Z$ be a centered Gaussian random variable with covariance matrix $\Sigma$. Then, $e^{\theta Z - n \hat\lambda(\theta)}$ is an analytic function of $\theta$, and we write the Taylor expansion 
   \begin{equation*}
       e^{\theta Z - \hat\lambda(\theta)} = 1 + \sum_{a \in \Z_{\geq 0}^d} \theta^a \psi_a(1,Z).
   \end{equation*}
   Taking expectation and applying Fubini's theorem, we see that \eqref{eq:our version of 3.17 in comets} must hold. As in \citep{MR2073332} and \citep[Lemma 3c]{MR1006293}, we write
   \begin{align}
           \phi_a(n,x) & := x^a + \hat\phi_a(n,x) = x^a + \sum_{|b|_1 + 2j \leq |a|_1, j \geq 1} A_a(b,j) x^b n^j \label{eq:phihat definition in definition of exponential martingales for random walk}\\
           \psi_a(n,x) & := x^a + \hat\psi_a(n,x) = x^a + \sum_{|b|_1 + 2j = |a|_1, j \geq 1} A_a(b,j) x^b n^j \label{eq:psihat definition in definition of exponential martingales for random walk}
   \end{align}
   In \eqref{eq:psihat definition in definition of exponential martingales for random walk}, only the $|b|_1 + 2j = |a|_1$ terms survive since only the second cumulants of the centered Gaussian are nonzero. In such terms, the coefficients $A_a(b,j)$ in both $\hat\psi_a$ and $\hat\phi_a$ match because the covariances of the underlying random walk $p$ and the covariances of the Gaussian are chosen to be equal. 
   Finally, we notice from \eqref{eq:psihat definition in definition of exponential martingales for random walk} that
   \begin{equation}
       \frac{1}{n^{|a|_1/2}} \hat\psi_a(n,x) = \hat\psi_a \left( 1, \frac{x}{\sqrt{n}} \right).
   \end{equation}
   Then,
   \begin{align*}
       \bigg\la \left( \frac{X_n}{\sqrt{n}} \right)^a \bigg\ra 
       & = \frac{1}{n^{|a|_1/2}} E\left[  \left(\phi_a\left( n,X_n \right) - \hat\phi_a\left( n,X_n \right) + \hat\psi_a\left( n,X_n \right) - \hat\psi_a\left( n,X_n \right) \right) \mu(\vx_n)  \right] \\
       & = \frac{1}{n^{|a|_1/2}} E\left[  \left(\phi_a\left( n,X_n \right) - \hat\phi_a\left( n,X_n \right) + \hat\psi_a\left( n,X_n \right) - \hat\psi_a\left( n,X_n \right) \right) \mu(\vx_n)  \right] \\
       & =  \frac{1}{W_n} \frac{M_n}{n^{|a|_1/2}} + \frac{1}{n^{|a|_1/2}} \bigg\la \left( \hat\phi_a \left( n,X_n \right) - \hat\psi_a \left( n,X_n \right) \right) \bigg\ra \\
       & \hphantom{=} - \bigg\la \hat\psi_a\left( 1,\frac{X_n}{\sqrt{n}} \right)  \bigg\ra \\
   \end{align*}
   The first term goes to $0$ from item \ref{it:second moment of martingale is finite} in Proposition \ref{prop:growth rate of polynomial martingales in weak disorder}. The second term is a polynomial with order strictly less than $|a|_1$. So by the induction assumption, each term in that polynomial converges to the Gaussian integral in \eqref{eq:gaussian limit for polynomial functions of endpoint}. Combining this with the observation in \eqref{eq:our version of 3.17 in comets}, we see that  
   \begin{equation}
       \lim_{n \to \infty} - \bigg\la \hat\psi_a\left( 1,\frac{X_n}{\sqrt{n}} \right)  \bigg\ra = \int x^a d\gamma_\Sigma(x).
   \end{equation}
   The integrand in the third term consists entirely of terms of the form $x^b n^j$ where $|b|_1 + 2j < |a|_1$. Thus, again applying \eqref{eq:gaussian limit for polynomial functions of endpoint} for powers of lower order than $|a|_1$, the third term goes to $0$. This completes the proof of the theorem. }
\end{proof}
\section{Strong disorder}
\label{sec:strong disorder}
In this section, we first complete the proof of Theorems \ref{thm:strong disorder statement} and \ref{thm:simultaneosly in weak and strong disorder}. The weak disorder part of Theorem \ref{thm:simultaneosly in weak and strong disorder} (item \ref{item:weak disorder statement}) has been proven in Theorem \ref{thm:weak disorder holds for small h}. The strong disorder part of Theorem \ref{thm:simultaneosly in weak and strong disorder} requires the entropy condition for strong disorder in Theorem \ref{thm:strong disorder statement}, which in turn relies on Proposition \ref{prop:entropy condition for strong disorder}, Lemma \ref{lem:as h goes to infinity the underlying random walk reduces}, and Lemma \ref{lem:simultaneous weak and strong disorder for finite range walks}. 

Proposition \ref{prop:entropy condition for strong disorder} gives a condition for strong disorder that compares a relative entropy of the weights with the Shannon entropy of the underlying random walk. It ought to be interpreted as saying that when the entropy of the weights (impurities) overwhelms the entropy of the underlying random walk (the thermal fluctuations), the polymer is in the strong disorder regime. Then, we prove Lemma \ref{lem:as h goes to infinity the underlying random walk reduces}, which shows that if $p$ has a finite logarithmic moment and the conditions of Theorem \ref{thm:strong disorder statement} are satisfied for some field $h$, then the behavior of the Shannon entropy of the random walk with kernel $q(\lambda h)$ as $\lambda \to \infty$ can be understood. Then, in Lemma \ref{lem:simultaneous weak and strong disorder for finite range walks}, we show that if $p$ has finite range, the favorable $h$ required by Lemma \ref{lem:as h goes to infinity the underlying random walk reduces} exist, and $\entropy(q(\lambda h)) \to 0$ as $\lambda \to \infty$. This shows that there are fields $h$ satisfying the entropy condition in Proposition \ref{prop:entropy condition for strong disorder}, thereby implying high temperature and strong disorder. 

Finally, we show that we cannot draw the same conclusions for Gaussian polymers, since for all $\beta$ small enough and external fields $h$, the polymer is in weak disorder. We also show explicitly that the entropy of the random walk $q(h)$ fluctuates periodically in $h$, is bounded uniformly above and below, and therefore, for all $\beta$ large enough \eqref{eq:gammatwo condition for general random walks} implies strong disorder for all $h \in \R^d$. 

\begin{prop}
   Suppose $\beta > 0$ is such that $\relativeHqp > \entropy(p)$, where $p$ is a (nontrivial) random walk transition probability. Then, the polymer is in the low temperature (and hence strong disorder) regime. 
   \label{prop:entropy condition for strong disorder}
\end{prop}
\begin{lem}
   Suppose there is a nonempty set $K \subset \Z^d$ such that $\argmax_{x \in \supp(p)} h \cdot {x} = K \subset \Z^d$. Then,
   \begin{equation}
       \lim_{\lambda \to \infty} \entropy(q(\lambda h)) = \entropy_K(p).
       \label{eq:entropy converges to Hk as lambda goes to infinite}
   \end{equation}
   \label{lem:as h goes to infinity the underlying random walk reduces}
\end{lem}
\begin{lem}
   Suppose $p$ is nontrivial and has finite range. Then, there exist $h$ such that $K(h)$ is a singleton, and hence $\entropy_{K(h)} = 0$.
   \label{lem:simultaneous weak and strong disorder for finite range walks}
\end{lem}
Using these three ingredients, we complete the proof of Theorem \ref{thm:strong disorder statement}, and use that to prove Theorem \ref{thm:simultaneosly in weak and strong disorder}. 
\begin{proof}[Proof of Theorem \ref{thm:strong disorder statement}]
   Fix $h \in \R^d$ such that $K(h)$ is nonempty. Lemma \ref{lem:as h goes to infinity the underlying random walk reduces} guarantees that $\entropy(q(\lambda h)) \to \entropy_{K(h)}(p)$ as $\lambda \to \infty$. Then, choose $\lambda$ so large, such that $\relativeHqp > \entropy(q(\lambda h))$, and it follows from Proposition \ref{prop:entropy condition for strong disorder} that the polymer is in the low temperature regime under the external field $\lambda h$.
\end{proof}

\begin{proof}[Proof of Theorem \ref{thm:simultaneosly in weak and strong disorder}]
    The proof of item \ref{item:weak disorder statement} is a direct consequence of Theorem \ref{thm:weak disorder holds for small h}. 

    Item \ref{item:strong disorder holds when p has finite range} follows from Lemmas \ref{lem:as h goes to infinity the underlying random walk reduces} and \ref{lem:simultaneous weak and strong disorder for finite range walks}. If $p$ is nontrivial and has finite range, then for every $\beta > 0$, there exist $h$ such that $K(h)$ is a singleton and hence $\entropy_{K(h)} = 0$. If $\supp(p) = \{z_1,\ldots,z_k\}$, then the set of $h$ such that $K(h)$ is not a singleton is contained inside the set where $\cup_{1 \leq i < j \leq k} \{h \colon h \cdot z_i = h \cdot z_j \}$ (see \eqref{eq:defining the codimension one bad set of fields}). It is clear that $K(\lambda h) = K(h)$ for all $\lambda > 0$, and since $\mgfrw_p(t) < \infty$ for all $t$, $E_p|\log(p(X_1))| < \infty$.

    Since $\relativeHqp$ is a relative entropy of measures, it is nonnegative. Moreover, it is $0$ if and only if $\Prob(\w)$ is a delta mass or if $\beta = 0$. Then, it follows from \eqref{eq:entropy converges to Hk as lambda goes to infinite} that for large enough $\lambda$, $\entropy(q_p(\lambda h)) < \relativeHqp$. Hence, the conditions of Theorem \ref{thm:strong disorder statement} are satisfied and $\lambda h \in \strong_\beta$.
\end{proof}

\begin{proof}[Proof of Proposition \ref{prop:entropy condition for strong disorder}]
    We mimic the classical Kahane-Peyriere argument~\citep{MR3444835} to obtain the entropic condition in \eqref{eq:gammatwo condition for general random walks} that ensures that the polymer is in the low temperature regime; i.e., strict inequality holds in the annealed bound.

Fix $\theta \in (0,1)$ and for this section alone, let us introduce the notation
\begin{equation*}
    Z_{i,j}(x) = E_{p} \left( \exp( \beta H(\vx_{i,j})) | X_i = x \right)
\end{equation*}
where $\vx_{i,j} = (X_i,\ldots,X_j)$ refers to the section of the random walk path from step $i$ to $j$. Using $(x+y)^\theta \leq x^\theta + y^\theta$ when $\theta \in (0,1)$, we get
\begin{align*}
   Z_{0,n}(\origin)^\theta
   & = \left( \sum_x p(x) Z_{1,n}(x) e^{\beta \w_{1,x}} \right)^{\theta}   \\
   & \leq \sum_x p(x)^\theta Z_{1,n}(x)^\theta e^{\theta \beta \w_{1,x}}. 
\end{align*}
Normalizing the partition function, taking expectation, and using independence, we get
\begin{align*}
    \E W_n^\theta \leq \sum_x p(x)^\theta \frac{\E[Z_{1,n}(x)^\theta]}{\mgf(\beta)^{(n-1)\theta}} \frac{\E[ e^{\theta \beta \w_{1,x}}]}{\mgf(\beta)^\theta} = r(\theta) \E[ W_{n-1}^\theta ],
\end{align*}
where
\begin{equation}
    r(\theta) = \frac{\mgf(\theta \beta)}{\mgf(\beta)^\theta} \sum_x p(x)^\theta.
    \label{eq:r-theta the function that appears in the strong disorder proof}
\end{equation}
We will show $r(\theta) < 1$ next, from which it will follow that 
\begin{equation*}
    \frac{1}{n} \log \E[ W_n^\theta ] \leq \log r(\theta) < 0.
\end{equation*}
Then, the Markov inequality and Borel-Cantelli lemma imply the almost sure exponential decay of $W_n$ to $0$. \notes{$\Prob(W_n > n^{-1}) \leq \E[ W_n^\theta] n^\theta$.}

To show $r(\theta) < 1$, we first show that the function $\log r(\theta)$ is convex. Since $\log r(\theta) = -\theta \logmgf(\beta) + \logmgf(\theta \beta) + \log  \sum_x p(x) \exp( (\theta - 1) \log p(x))$, it is a sum of a term linear in $\theta$, and two terms of the form $\log E[ \exp( \theta X + Y) ]$, where $E$ is either $\E$ or $E_p$, and $X$ and $Y$ are the appropriate random variables. Differentiating with respect to $\theta$, we see that 
\begin{align*}
    \frac{d}{d\theta} \log E[ \exp( \theta X + Y) ] & = \frac{E[ X \exp( \theta X + Y) ]}{E[ \exp( \theta X + Y) ]},\\
    \frac{d^2}{d\theta^2} \log E[ \exp( \theta X + Y) ] & = \frac{E[ X^2 \exp( \theta X + Y) ]}{E[ \exp( \theta X + Y) ]} - \frac{E[ X \exp( \theta X + Y) ]^2}{E[ \exp( \theta X + Y) ]^2}. 
\end{align*}
Thus the first and second derivatives are the mean and variance of the random variable $X$ with respect to a weighted measure, and this proves that the second derivatives are nonnegative. Note that $\log r(0) = \log \sum_x p(x)^0 = \log (\supp(p)) > 0$ if the underlying random walk is nontrivial, and that $\log r(1) = 0$. Thus, we would only be able to find $\theta$ such that $r(\theta) < 1$ if $d \log r(\theta)/d\theta |_{\theta = 1} > 0$. This gives us the condition
\begin{align*}
    \frac{d}{d\theta} \log r(\theta) \big|_{\theta = 1} 
    & = \beta \logmgf'(\theta \beta) - \logmgf(\beta) + \frac{\sum_x p(x)^\theta \log p(x)}{\sum_x p(x)^\theta} \Bigg|_{\theta = 1}\\
    & = \beta \logmgf'(\beta) - \logmgf(\beta) - \entropy(p)=\relativeHqp-\entropy(p)>0,
\end{align*}
where in the last equality we used \eqref{eq:gammatwo quantity used in strong disorder condition}.
In the standard nearest-neighbor random walk case, $\entropy(p) = \log 2d$, and this gives the classical condition for low temperature (see Proposition \ref{prop:condition on gamma 2 for strong disorder}). 
\end{proof}

\begin{proof}[Proof of Lemma \ref{lem:as h goes to infinity the underlying random walk reduces}]
   Let $k = \max_{x \in \supp(p)} h \cdot x$. Then, the dominated convergence theorem gives us that
   \begin{align*}
       \lim_{\lambda \to \infty} E_p[ e^{\beta \lambda (h\cdot X_1 - k)} 1_{ X_1 \in K^c } ] & = 0.
   \end{align*}
   The latter can be written more succinctly as $E_p[ e^{\beta \lambda h \cdot X_1} ] = e^{\beta \lambda k} p(K)(1 + o(1) )$. Then, as $\lambda \to \infty$,
   \begin{align*}
      \entropy(q(\lambda h))
      & = - E_p \left[ \frac{e^{\beta \lambda h \cdot X_1}}{E_p[ e^{\beta \lambda h \cdot X_1} ]} \log\left(  p(X_1) \frac{e^{\beta \lambda h \cdot X_1}}{E_p[ e^{\beta \lambda h \cdot X_1} ] }  \right) \right],\\
      & = - E_p \left[ \frac{1}{( p(K) + o(1) )} \left( \log \frac{p(X_1)}{p(K)} + o(1) \right) 1_{X_1 \in K} \right] \\
      & \quad - E_p \left[ \frac{e^{\beta \lambda (h \cdot X_1 - k)}}{( p(K) + o(1) )} \left( \log \frac{p(X_1)}{p(K)} + o(1) \right) 1_{X_1 \in K^c} \right] \\
      & \quad - E_p \left[ \frac{e^{\beta \lambda (h \cdot X_1 - k)}}{( p(K) + o(1) )}   \beta \lambda (h \cdot X_1 - k) 1_{X_1 \in K^c} \right] \\
      & = \entropy_K(p) + o(1),
  \end{align*}
  where in the last equality we have applied the dominated convergence theorem using the integrability of $|\log p(X_1)|$ for in the second line, and the boundedness of the function $ze^{-z}$ when $z\geq 0$ for the last term. Note that the integrability of $|\log p(X_1)|$ follows from \eqref{eq:p(X_1)_finite_moment} since \eqref{eq:p(X_1)_finite_moment} shows that $\log p(X_1)$ has a moment generating function of positive radius.
\end{proof}

\begin{proof}[Proof of Lemma \ref{lem:simultaneous weak and strong disorder for finite range walks}]
    Fix $h = (h_1,h_2,0,\ldots,0)$ such that $h_1/h_2$ is a finite irrational. Since $\supp(p)$ has finite cardinality, there must be a unique $x \in \supp(p)$ where $h \cdot x$ takes its maximum. Suppose not; then, there must exist integers $z_1,z_2$ (wlog $z_2 \neq 0$) such that $\sum_{i=1}^2 h_i z_i = 0$, or $z_2 = -(h_1/h_2)z_1$. This is a contradiction since $h_1/h_2$ was assumed to be irrational. Therefore, the cardinality of $K(h)$ is $1$, and consequently $\entropy_{K(h)}(p) = 0$. 
\end{proof}

\subsection{Gaussian random walk}
\label{sec:gaussian random walk}

In this section, we discuss the case where $p(x)=Ce^{-|x|^2/2}$ for $x \in \Z^d$, $d\geq 3$, where the normalization constant $C^{-1}=\sum_xe^{-|x|^2/2}$ is chosen so $p$ is a probability distribution.
\begin{proof}[Proof of Proposition \ref{prop:gaussian random walk polymer in weak disorder for all h}]
 By Theorem \ref{thm:original l2 theorem for polymer}, if we show that there exists $c<1$ such that $\pi(q(h))<c$ for all $h$, it will follow that if $\beta$ is chosen so that $\mgf_2(\beta)<c^{-1}$, the polymer will be in the weak disorder regime for all $h$.

The quantity $\pi(q(h))$ only depends on the distribution of $X-X'$, where $X$ and $X'$ are independent random variables with distribution $q(h)$. We have

\begin{align*}
P_{q(h)}(X-X'=z)&=\sum_{w\in\Z^d}\frac 1{C_h}e^{-\frac{|z+w|^2}{2}+\beta h\cdot (z+w)} \frac 1{C_h}e^{-\frac{|w|^2}{2}+\beta h\cdot w}
\\&=\frac{1}{C_h^2}e^{|\beta h|^2}e^{-\frac{|z|^2}2}\sum_{w\in\Z^d}e^{-\left|w-\left(\beta h-\frac z2\right)\right|^2}
\\&=C'_he^{-\frac{|z|^2}2}\sum_{w\in\Z^d-\beta h}e^{-\left|w+\frac z2\right|^2},
\end{align*}
where $C_h$ is the normalization constant to get a probability measure and $C'_h={e^{|\beta h|^2}}/{C_h^2}$. Notice that the sum in the last display depends on $h$ only through the elementwise fractional part $\{\beta h\}$ of the vector $\beta h$, so the distribution of $X-X'$ depends only on $\{\beta h\}$ as well. In the proof of Theorem \ref{thm:weak disorder holds for small h}, it is shown that $\pi(q(h))$ is continuous in $h$ (the argument there shows continuity at $0$, but changing the base measure from $p$ to $q(h_0)$ gives continuity at all $h_0$'s as well). Since $d\geq 3$, for any fixed $h$ we have $\pi(q(h))<1$. By continuity of $\pi(q(h))$ and compactness of $\beta^{-1} [-1,1]^d$, we get a constant $c<1$ such that $\pi(q(h))<c$ for all $h\in \beta^{-1} [-1,1]^d$. Since $\pi(q(h))$ depends on $h$ only through the fractional part $\{\beta h\}$, we get $\pi(q(h))<c<1$ for all $h\in\R^d$ as needed.

To show the statement about strong disorder, we obtain explicit entropy estimates for any dimension $d$. First, we obtain them in $d=1$, and note that for $d \geq 1$, since $p$ is of the form $p(x)=p_1(x_1)\dots p_d(x_d)$, where $p_1,\dots,p_d$ are $d$ copies of the $1$-dimensional discrete Gaussian, $\entropy(p)=\entropy(p_1)+\dots +\entropy(p_d)=d \entropy(p_1)$. In other words, entropy tensorizes over product measures.

Setting $t=\lambda h$, $C_0(t)=\sum_{x\in\mathbb{Z}-t} e^{-x^2/2}$ and $C_2(t)=\sum_{x\in\mathbb{Z}-t} x^2e^{-x^2/2}$, we have
$$
    E_p(e^{\lambda h  X })=C\sum_{x\in\mathbb{Z}}e^{tx-x^2/2}=Ce^{t^2/2}\sum_{x\in\mathbb{Z}}e^{-(x-t)^2/2}=Ce^{t^2/2}C_0(t).
$$
Thus,
\begin{align}
\entropy(q(\lambda h))=& -\sum_x\frac{e^{\lambda h  x }}{E_p(e^{\lambda h  X })}p(x)(\ln p(x)+\lambda h  x -\ln E_p(e^{\lambda h  X }))
\nonumber \\=&\ln C_0(t)-\sum_x\frac{e^{-(x-t)^2/2}}{C_0(t)}\left(-\frac{(x-t)^2}2\right)
\nonumber \\=&\ln C_0(t)+\frac{C_2(t)}{2C_0(t)}. \label{eq:exact entropy expression for gaussian random walk}
\end{align}

Note, that the dependence of $C_0$ and $C_2$ on $t$ is only in the set over which they are summed. If $t$ ranges over the integers, we have $\mathbb{Z}-t=\mathbb{Z}$, so $C_0(t)$ and $C_2(t)$ are constants independent of $t$. Both are super-exponentially decaying series, and so can easily be approximated by the first few terms. In fact, both are approximately equal to $2.5$, giving $\entropy(q(t))\approx 1.4$. 

For non-integer $t$'s we estimate as follows. First, both $C_0(t)$ and $C_1(t)$ are periodic with period $1$, so for our estimates, without loss of generality, we assume $0<t<1$. In fact, $C_0(t)=C_0(1-t)$ and $C_2(t)=C_2(1-t)$, so we can further assume $0<t\leq 1/2$. For $n \in \Z$, using the inequalities $e^{-(n+t)^2/2}\geq e^{-n^2/2}$ for $n\leq -1$, and $e^{-(n+t)^2/2}>e^{-(n+1)^2/2}$ for $n\geq 2$, we obtain 
$$C_0(t)\geq e^{-t^2/2}+e^{-(t+1)^2/2}-e^0-e^{-1/2}-e^{-2^2/2}+C_0(0).$$ 
The right-hand side is minimized when $t=1/2$ giving $C_0(t)>C_0(0)-0.54$. Similarly, we can get $C_0(t)<C_0(0)+0.51$ giving us $1.96<C_0(t)<3.04$.  Arguing in the same way, we get the bounds $2.08<C_2(t)<2.93$, and therefore for $d \geq 1$, using \eqref{eq:exact entropy expression for gaussian random walk} and the tensorization of the entropy, we get  
$$1.01d<\entropy(q_p(t))<1.85d.$$
It follows that if $\beta$ is small enough such that $\relativeHqp <1.01d$, Theorem \ref{thm:strong disorder statement} cannot guarantee the existence of strong disorder.

However, the proof of Corollary 2.1 of \cite{MR3444835} shows that $\relativeHqp \to \infty$ as $\beta \to \infty$. Therefore, for all large enough $\beta$ and all $h \in \R^d$, the condition in \eqref{eq:gammatwo condition for general random walks} holds and implies strong disorder. 
\arjun{Originally, we had explicit computations for all the constants. We  establish $C_2(0)+0.43>C_2(t)>C_2(0)-0.41$ just like we did for $C_0(t)$. Combining these estimates, we get that for any $t$ we have $$1.01<\ln 1.96+\frac{2.08}{2\cdot 3.04}<\entropy(q_p(t))<\ln 3.04+\frac{2.93}{2\cdot 1.96}<1.85.$$}
\arjun{However, if $\sup_t \entropy(q_p(t)) < \entropy(\mathbb{Q}_\beta | \Prob)$, can we ensure the existence of simultaneous strong and weak disorder?} 
\end{proof}

\section{Monotonicity}
\label{sec:monotonicity}
In this section, we prove Theorem \ref{thm:monotonicity}, which says that if $h$ $\in \strong_{\beta_0}$, then $h\beta_0/\beta \in \strong_\beta$ for all $\beta > \beta_0$. Recall the definition of $\gpl{h}$ in \eqref{eq:gpl definition} and $\logmgfrw(\beta h)$ just below \eqref{eq:annealed bound for gpl}. Let 
\begin{equation}
   \hat{g}_n(\beta,h) = \frac{1}{n} \log E_{q(h/\beta)} \left[ \exp( \beta H(\vx_n) ) | X_0 = \origin \right] + \logmgfrw(h).
   \label{eq:ghat definition in the monotonicity proof}
\end{equation}
Note that $q(h/\beta)$ is actually independent of $\beta$, and the only dependence on $\beta$ in the first term in \eqref{eq:ghat definition in the monotonicity proof} is in $\beta H(\vx_n)$. It is clear that $\lim_{n \to \infty} \hat{g}_n(\beta, h) =: \hat{g}(\beta,h) = \beta \gpl{h/\beta}$ in $L^1$ and $\Prob-\almostsurely$.
\begin{prop}
    The function $\beta \mapsto \hat{g}(\beta,h) - \logmgf(\beta) - \logmgfrw(h)$ is nonincreasing on $\R_{\geq 0}$ and nondecreasing on $\R_{\leq 0}$.
    \label{prop:monotonicity result using for gpl}
\end{prop}
The proof of the proposition follows the clever application of the FKG inequality used in \cite[Theorem 3.2(b)]{MR2271480}.

\begin{proof}
We will show that for fixed $h$, $\beta \to \E[ \hat{g}_n(\beta,h)] - \logmgf(\beta) - \logmgfrw(h)$ is nonincreasing on $\R_{\geq 0}$ and nondecreasing on $\R_{\leq 0}$. Taking $n \to \infty$ produces the result we need. Differentiating
with respect to $\beta$, we get
\begin{equation}
    \E[ \partial_\beta \hat{g}_n(\beta,h) ] = \frac1{n} \E \left[ \frac{ E_{q(h/\beta)}[ H(\vx_n) \exp(\beta H(\vx_n))]}{Z_{n,\beta,q(h/\beta)}} \right].
    \label{eq:equation for partial deriv of hat g wrt beta}
\end{equation}
For any fixed random walk path $\vx_n$, let the probability measure $\hat{\Prob}(\{\w_x\}_{x \in \Z^d})$ be the measure defined by
\begin{equation}
    \int f(\w) \hat{\Prob}(d \w) = \int f(\w) \exp(\beta H(\vx_n) - n \logmgf(\beta)) \Prob(d \w),
\end{equation}
and let $\hat{\E}$ be the corresponding expectation.

It is easy to see that $\hat{\Prob}$ is a product measure. Note that when $\beta>0$, $H(\vx_n)$ and $Z_{n,\beta,q(h/\beta)}$ are both increasing functions of the weights $\{\w_x\}_{x \in \Z^d}$. Then, using the FKG inequality for the measure $\hat{\Prob}$ as in \citep[pp.\,1755]{MR3444835}, we get
\begin{align*}
   E_{q(h/\beta)} \E \Bigg[ & \frac{  H(\vx_n) \exp(\beta H(\vx_n) - n\logmgf(\beta))}{Z_{n,\beta,q(h/\beta)}} \Bigg] \mgf(\beta)^n  
   =  E_{q(h/\beta)} \hat{\mathbb{E}} \left[ \frac{  H(\vx_n) }{Z_{n,\beta,q(h/\beta)}} \right]  \mgf(\beta)^n \\
   & \leq E_{q(h/\beta)} \bigg[ \E [ H(\vx_n) \exp(\beta H(\vx_n) - n\logmgf(\beta))]  \\ 
   & \qquad \times \E [ Z_{n,\beta,q(h/\beta)}^{-1} \exp(\beta H(\vx_n) - n\logmgf(\beta))]  \bigg] \mgf(\beta)^n\\
   & =  n \logmgf'(\beta). 
\end{align*}
Plugging this back into \eqref{eq:equation for partial deriv of hat g wrt beta}, we see that
\begin{equation*}
    \partial_\beta (\E[ \hat{g}_n(\beta,h)] - \logmgf(\beta) -\logmgfrw(h)) \leq 0.
\end{equation*}
In the case $\beta<0$, the partition function $Z_{n,\beta,q(h/\beta)}$ is a decreasing function of the weights $\{\w_x\}_{x \in \Z^d}$ so the inequalities above are reversed. Taking $n \to \infty$ completes the proof.
\end{proof}

\begin{proof}[Proof of Theorem \ref{thm:monotonicity}]
   Suppose $(h,\beta_0)$ is in low temperature for some $\beta_0 > 0$. Then,
       $\hat{g}(\beta_0,h \beta_0) < \logmgf(\beta_0) + \logmgfrw(h \beta_0),$ and it
   follows from Proposition \ref{prop:monotonicity result using for gpl} that
   \begin{equation}
       \hat{g}(\beta,\beta_0 h) < \logmgf(\beta) + \logmgfrw(\beta_0 h) \quad \forall \beta > \beta_0.
   \end{equation}
   This implies that $\beta \gpl{\beta_0 h/\beta} < \logmgf(\beta) + \logmgfrw(\beta_0 h)$, and thus $(\beta,\beta_0 h/\beta)$ is in low temperature.
\end{proof}

\section{Numerical Results}
\label{sec:experiments}
In this section, we present some numerical experimental results in the $d=1$ and $d=3$ cases where the weights are iid $\Uniform[0,1]$ random variables, $p$ is the nearest-neighbor simple random walk, and $\beta=1$ throughout. Figure \ref{fig:gpl for different values of beta} clearly showed the phase transition in $d=3$ as $\gpl{h}$ diverged away from the annealed bound when $h$ was increased. We were curious to see if the phase transition appeared at the level of fluctuations in numerical simulations. So we simulated the growth exponents for the standard deviation of the log partition function and several other quantities related to the fluctuation of the endpoint. They are
\begin{enumerate}
    \item $\sqrt{\mathbb E [ \log(Z_n)^2] - \mathbb E[\log(Z_n)]^2}$, the variance of the log partition function.  
    \item $\E \left[  \la |X_n|^2\ra - \la |X_n|\ra^2  \right]$, the annealed variance of the endpoint under the Gibbs measure.
    \item $\sqrt{\E [ \la |X_n|_2^2 \ra]  - \E[ \la |X_n| \ra]^2}$, the standard deviation of the endpoint under the annealed Gibbs measure. 
    \item \label{item:argmax of the annealed gibbs measure} $\sqrt{\E[|A_n|^2] - \E[|A_n|]^2 }$, where $A_n = \argmax_{|z|_1 = n} \la 1_{X_n = z} \ra$. This is the standard deviation of the location of the maximum of the Gibbs measure. 
\end{enumerate}
We simulated these statistics for four different parameter combinations.
\begin{itemize}
\item $d=1, h=0, n=10000, 1000$ samples of the environment
\item $d=1, h=3, n=10000, 1000$ samples of the environment
\item $d=3, h=0, n=400, 100$ samples of the environment
\item $d=3, h=3, n=400, 100$ samples of the environment
\end{itemize}
The simulations for $d=3$ were constrained by memory and runtime. If $t(n)$ is the quantity of interest, we determined the growth exponent by fitting a line to the graph of $\log(t(k))$ versus $\log(k)$ for $k=1,\ldots,n$ and determining the slope. The results appear in Table \ref{table:various endpoint and partition function fluctuation statistics}.
\begin{table}
\begin{center}
\begin{tabular}{ |c|cc|cc|c| } 
\hline
\multirow{2}{*}{Statistic} & \multicolumn{2}{c|}{$d=1$} & \multicolumn{2}{c|}{$d=3$} & \multirow{2}{*}{Exponent} \\
\cline{2-5}
 & $h=0$  & $h=3$ & $h=0$  &  $h=3$ & \\ 
 \hline
 $\sqrt{\mathbb E [ \log(Z_n)^2] - \mathbb E[\log(Z_n)]^2}$  & 0.32 & 0.35 & 0.03  & 0.05 & $\chi$ \\ 
 \hline
 $\sqrt{\E \left[  \la |X_n|^2\ra - \la |X_n|\ra^2  \right]}$ & 0.51  & 0.49 & 0.49 & 0.51 & 1/2 \\ 
 \hline
$\sqrt{\E [ \la |X_n|_2^2 \ra]  - \E[ \la |X_n| \ra]^2}$   & 0.59 &  0.63  & 0.49 & 0.51 & $\xi$ \\ 
 \hline
$\sqrt{\E[|A_n|^2] - \E[|A_n|]^2 }$   & 0.65 &  0.67  & 0.41 & 0.46 & $\xi$ \\ 
 \hline
\end{tabular}
\end{center}
\caption{Various endpoint and partition function fluctuation statistics}
\label{table:various endpoint and partition function fluctuation statistics}
\end{table}

The variance of the log partition function is thought to grow like $n^{\chi}$, where $\chi$ is called the fluctuation exponent. For the special log-gamma polymer in $d=1$, $\chi$ is known to be  $2/3$ \citep{MR2917766}, but this is expected to be true in much wider generality; this is the KPZ universality conjecture. For $d \geq 3$, when $\beta$ is in the $L^2$ portion of the weak-disorder regime, it is clear that the fluctuations are $O(1)$, and thus $\chi = 0$. Little is known or conjectured about $\chi$ outside of these regimes.

In the physics literature, \citet{huse_henley_1985} conjectured that the annealed variance of the endpoint under the Gibbs measure grows linearly. This was also seen in the numerical experiments in \citep{MR4402225}. Our results show that this is indeed the case in  $d=1$ and in $d=3$ in both the weak and strong disorder regimes. This is related to the non-zero and bounded curvature of $\gpl{h}$.

In $d=1$, again, physical arguments and numerical simulations suggest that the standard deviation of the argmax of the Gibbs measure grows with transversal fluctuations exponent $\xi = 2/3$. Related quantities have been computed for the special log-gamma polymer: Sepp\"al\"ainen considered the point-to-point polymer, and showed that the transversal fluctuation exponent at intermediate times is bounded above by $2/3$ \citep[Theorem 2.5]{MR2917766}. Some fluctuation results have also been proved for the half-space log-gamma polymer \citep{barraquand2023kpz}.

In weak disorder in $d \geq 3$, these fluctuations ought to be diffusive and grow with exponent $1/2$ at least when $h=0$, which certainly ought to be in the weak disorder phase. However, our simulated growth exponents are a bit smaller than $1/2$. This is probably because $n$ and the number of samples are both too small, and so these results ought not to be taken seriously.

Given the results of item \ref{item:argmax of the annealed gibbs measure} in $d=1$, if the annealed Gibbs measure on the endpoint assigns an $O(1)$ probability to a region around its maximum, the standard deviation of the endpoint under the annealed Gibbs measure should fluctuate with at least exponent $2/3$. However, our results (especially the one for $d=1,h=0$) suggest that the standard deviation grows with an exponent significantly smaller than $2/3$, which implies that the annealed Gibbs measure on the endpoint concentrates quite sharply around its maximum. 

The results for $d=3$ suggest that the standard deviation of the endpoint under the annealed Gibbs measure is diffusive, which is pretty counterintuitive. This again suggests that $n=500$ is probably too low for the purposes of computing the fluctuation exponents in $d=3$.
\graphicspath{ {./Figures/} }

\printbibliography


\end{document}